\documentclass[12pt, oneside]{article}   	
\usepackage{geometry}
\usepackage{graphicx}				

\usepackage{indentfirst}
\usepackage{hyperref}
\usepackage{tikz-cd}
\usepackage{ulem}	
\usepackage{amsthm}
\usepackage{amsfonts}
\usepackage{amsmath}
\usepackage{amscd}
\usepackage{enumerate}
\usepackage{amssymb}	
\usepackage{verbatim}
\usepackage{cancel}
\usepackage{cite}
\usepackage{textcomp}								
\usepackage{color}
\usepackage{mathrsfs}
\usepackage{appendix}
\usepackage[nottoc]{tocbibind}

\makeatletter
\newcommand{\subjclass}[2][2020]{\def\@subjclass{#2}\def\@subjclassyear{#1}}
\newcommand{\keywords}[1]{\def\@keywords{#1}}

\newcommand{\printinfo}{\let\oldthefootnote\thefootnote
  \renewcommand{\thefootnote}{} \footnotetext{\hspace{-1.8em}\begin{minipage}[t]{\textwidth}
      \small
      \textbf{\@subjclassyear\ Mathematics Subject Classification.} \@subjclass. \\
      \textbf{Keywords.} \@keywords.
    \end{minipage}
  }
  \let\thefootnote\oldthefootnote
  \addtocounter{footnote}{-1}
}
\makeatother

\newtheorem{example}{Example}[section]							
\newtheorem{theorem}{Theorem}[section]
\newtheorem{definition}{Definition}[section]
\newtheorem{proposition}{Proposition}[section]
\newtheorem{lemma}{Lemma}[section]
\newtheorem{remark}{Remark}[section]

\newtheorem{corollary}{Corollary}[section]

\def\rd{\mathrm d}

\def\top{\operatorname{top}}

\def\vol{\operatorname{vol}}

\def\be{\begin{align}}
\def\ee{\end{align}}

\title{Bernoulli property for skew products of exponentially mixing systems over Anosov bases}
\author{Changguang Dong, Qiujie Qiao}
\date{\today}

\begin{document}

\maketitle

\printinfo

\begin{abstract}
We study $C^{1+\beta}$ skew products over topologically mixing Anosov diffeomorphisms, 
with fiber dynamics generated by a locally free $\mathbb{R}^k$-action. 
For a Gibbs measure on the base and a cocycle with a nonzero average, 
we establish a Bernoulli criterion for ergodic invariant product measures. 
The fiber hypotheses apply to actions that are exponentially mixing with respect to the SRB measure, 
as well as volume measures that are almost exponentially mixing with a limit SRB measure. 

As an application, 
when the fibers are generated by geodesic flows or Weyl chamber actions, 
our results reveal a statistical phase transition determined by the drift. 
\end{abstract}

\tableofcontents

\section{Introduction}\label{sect:Main results}

Smooth dynamical systems exhibit a wide range of ergodic and statistical properties. 
Among these, the Bernoulli property, exponential mixing, 
and the Central Limit Theorem (CLT) are three central features
describing the statistical behavior of chaotic systems.
Exponential mixing controls the decay of correlations of regular observables, 
while the CLT describes the fluctuations of their Birkhoff sums. 
In the smooth setting, the latter two are quantitative properties 
that leverage the underlying smooth structure. 
In contrast, the Bernoulli property is one of the strongest measure-theoretic 
notions of randomness and does not require such additional structure.

Although exponential mixing is a quantitative property,
its connection to stronger ergodic properties long remained non-trivial. 
In \cite{Kan18}, Kanigowski used exponential mixing to establish
the Bernoulli property for $K$-property on semisimple
homogeneous spaces associated with irreducible lattices.
Subsequently, in \cite{DKRH24}, Dolgopyat, Kanigowski, and Rodriguez Hertz
established that for smooth measures, exponential mixing implies
the Bernoulli property.
More recently, Maldonado \cite{Mal25} obtained the corresponding
implication for SRB measures.
He also established the Bernoulli property for nonatomic
limit SRB measures arising from almost exponentially mixing
volume, in the setting of the exponential volume limits
studied by Ben Ovadia and Rodriguez Hertz \cite{BORH23-2}. 
These results reveal a strong connection between 
quantitative mixing and the Bernoulli property 
under suitable assumptions on the invariant measure. 

Skew products provide a natural setting for this investigation.
The statistical behavior of such systems depends not only on the
dynamics of the base and fiber factors,
but also on the interaction governed by the cocycle.

For compact group extensions of hyperbolic diffeomorphisms,
Dolgopyat \cite{Dol02} related mixing rates to properties of
stable and unstable holonomies and showed that the
decay of correlations is generically faster than every polynomial.
Field, Melbourne, and T\"or\"ok \cite{FMT03} obtained
exponential correlation decay for equivariant H\"older
observables on certain group extensions 
of hyperbolic basic sets, together with limit theorems
under the corresponding hypotheses. 
For circle extensions of piecewise $C^2$ uniformly expanding
covering maps,
Butterley and Eslami \cite{BE17} proved exponential mixing
when the cocycle is not Lipschitz cohomologous to a
piecewise constant function.
Recently, Doucette studied the Bernoulli property for skew products on 
$M \times \operatorname{PSL}_2(\mathbb{R})/\Gamma$, 
whose proof does not require the base dynamics to be exponentially mixing \cite{Dou26}. 

Let $f\colon X\to X$ be a diffeomorphism of a closed manifold,
let $\{\rho_{\mathbf t}\}_{\mathbf t\in\mathbb R^k}$
be a smooth $\mathbb R^k$-action on a closed manifold $Y$,
and let $\tau\colon X\to\mathbb R^k$ be a continuous cocycle.
We define the skew product $F$ as
$$
F(x,y)=\left(f(x),\rho_{\tau(x)}y\right).
$$
If $\mu$ is an $f$-invariant probability measure and $m$
is invariant under the fiber action, then $\mu\times m$
is $F$-invariant.
We call $\mu(\tau):=\int_X\tau\,\mathrm d\mu$ the drift of the cocycle.

In \cite{DDKN22-1,DDKN22-2},
Dolgopyat, Dong, Kanigowski, and N\'andori
systematically studied the ergodic and statistical properties
of smooth skew products with an $\mathbb R^k$-action on the fiber.
In particular, \cite[Theorems 1.5 and 2.7]{DDKN22-2}
provide zero-drift examples over Anosov bases with
positive-entropy homogeneous fiber dynamics.
By choosing the dimension appropriately, 
the authors obtained, for each prescribed $m\in\mathbb N$, 
examples with correlation decay $O(n^{-m})$ that satisfy the $K$-property 
and the CLT but are not Bernoulli. 
These examples show that even rapid polynomial decay of
correlations, together with the $K$-property and the CLT,
does not ensure the Bernoulli property. 

In the zero drift setting, 
a different behavior occurs for certain zero entropy fibers. 
The first named author and Kanigowski \cite{DK22}
established a Bernoulli criterion
for extensions of hyperbolic systems by quasi-elliptic
zero entropy flows,
under suitable assumptions on an aperiodic zero-mean cocycle
and on the fiber flow.
Their applications include
typical translation flows on surfaces and suitable smooth
reparametrizations of isometric flows on $\mathbb T^2$.  
These results show that 
the fiber dynamics and their interaction with the cocycle
are also essential. 
This naturally raises the question of whether 
positive entropy fiber dynamics with nonzero drift 
possess the Bernoulli property. 

In this paper, we investigate the statistical properties of the invariant
product measure $\mu\times m$, 
where the base map $f$ is a topologically mixing Anosov diffeomorphism, 
$\mu$ is a Gibbs equilibrium state of a H\"older potential, 
and the cocycle exhibits a nonzero drift ($\mu(\tau)\ne\mathbf 0$).
Since $\mu$ need not be an SRB measure, the product measure fails 
to satisfy the smoothness or SRB hypotheses of previous criteria,
even when the invariant fiber measure $m$ is smooth or SRB. 
Consequently, those criteria do not directly apply to our setting.

We establish a Bernoulli criterion for such systems. 
Under the stated exponential mixing or limit SRB hypotheses,
our approach yields the Bernoulli property 
without requiring the base measure to be SRB. 
As an application, we investigate nonzero drift extensions with 
geodesic flows or Weyl chamber actions in the fiber 
(the same dynamics studied in \cite{DDKN22-2}). 
Furthermore, we prove that systems with nonzero drift are Bernoulli, 
exhibit exponential mixing of all orders and satisfy the CLT, 
revealing a statistical phase transition as discussed in Remark \ref{remark:Statistical Phase Transition}. 

\subsection{Main results}\label{subsect: Main results}

We first specify the mixing assumptions used in our results.

When $0<r<1$, $C^r(Y)$ denotes the H\"older space with exponent $r$, equipped with the norm $\|\varphi\|_{C^r}:=\|\varphi\|_{C^0}+|\varphi|_{C^r}$, where 
\begin{align*}
\left\|\varphi\right\|_{C^0}:=\sup_{y\in Y} |\varphi(y)|,\quad 
|\varphi|_{C^r}:=\sup_{x,y\in Y, \, x\neq y} \frac{|\varphi(x)-\varphi(y)|}{|x-y|^r}. 
\end{align*}
When $r\in\mathbb{N}$, $C^r(Y)$ denotes the standard $C^r$ function space equipped with the standard $C^r$ norm $\|\cdot\|_{C^r}$. 
When $r>1$ is a non-integer, $C^r(Y)$ is the space consisting of $C^{\lfloor r \rfloor}$ functions whose $\lfloor r \rfloor$-th derivatives are H\"older continuous with exponent $r-\lfloor r \rfloor$, equipped with the norm $\|\cdot\|_{C^r} = \|\cdot\|_{C^{\lfloor r\rfloor}} + |\cdot|_{C^{r-\lfloor r\rfloor}}$. 

\begin{definition}\label{def:exponentially mixing}
Let $\{\rho_{\bf t}\}_{{\bf  t}\in\mathbb{R}^k}$ be an $\mathbb{R}^k$-action on a compact metric space $Y$, and let $\nu$ denote its invariant measure. 
The action is said to be \textbf{exponentially mixing} if the inequality 
\begin{align}\label{DDKN22-1-(3.1)}
\left| \int B_1(y)B_2(\rho_{\bf t} y)\,\rd\nu(y) - \nu(B_1)\nu(B_2) \right|
\leq Ce^{-\delta_0\|{\bf t}\|} \|B_1\|_{C^r}\|B_2\|_{C^r}
\end{align}
holds for all $B_1, B_2 \in C^r(Y)$ for some $r>0$, $C>0$, and $\delta_0>0$. 

Furthermore, the action $\{\rho_{\bf t}\}$ is said to be \textbf{exponentially mixing of order $s$} if
$$
\left|
\nu\left( \prod_{j=1}^s B_j\big(\rho_{{\bf t}_j} y\big) \right)
- \prod_{j=1}^s \nu(B_j)
\right|
\leq C e^{-\delta\,l({\bf t}_1,\dots,{\bf t}_s)}
\prod_{j=1}^s \|B_j\|_{C^r}
$$
holds for all $B_1, \dots, B_s \in C^r(Y)$ for some $r>0$, $C>0$, and $\delta>0$, where 
$$
l({\bf t}_1,\dots,{\bf t}_s) = \min_{i\neq j} \big\|{\bf t}_i - {\bf t}_j\big\|.
$$
The action is {\bf exponentially mixing of all orders} 
if it is exponentially mixing of order $s$ for every integer $s\ge 2$. 
The above definitions extend naturally to smooth diffeomorphisms.
\end{definition}

In the exponentially mixing setting, if the invariant measure $\nu$ is smooth, 
then by \cite[Lemma B.1]{DKRH24}, the validity of \eqref{DDKN22-1-(3.1)} 
for some $r>0$ implies that it holds for all $r'>0$ 
(with constants $C$ and $\delta$ possibly depending on $r'$). 

\begin{definition}\label{def:almostmixing}
Let $\{\rho_{\bf t}\}$ be a $\mathbb{R}^k$-action on a closed manifold $Y$, 
and let $m_{\vol}$ denote the normalized Riemannian volume measure on $Y$, 
not necessarily invariant. 
We say that \textbf{volume is almost exponentially mixing} for the action $\{\rho_{\bf t}\}$ 
if there exists $r>0,C,\delta_0>0$ such that for all $g,h\in C^r(Y)$ 
with $\int h\,\rd m_{\vol}=0$ and for all $\mathbf t\in\mathbb R^k$, 
$$
\left|\int g\circ \rho_{\bf t}\, h\,\rd m_{\vol}\right|
\le Ce^{-\delta_0 \|{\bf t}\|} \|g\|_{C^r}\|h\|_{C^r}.
$$
The above definitions extend naturally to arbitrary smooth diffeomorphisms.
\end{definition}

For a $C^{1+\beta}$ diffeomorphism $f$ on a closed manifold $X$, 
Ben Ovadia and Rodriguez Hertz proved in \cite{BORH23-2} that 
if the volume measure $m_{\vol}$ is almost exponentially mixing, 
there exists an $f$-invariant Borel measure $\mu$, 
obtained as a limit point of the sequence $\{f^n_* m_{\vol}\}_{n \in \mathbb{N}}$, 
such that for all $g, h \in C^r(X)$:
\begin{align}\label{Mal25-(1.2)}
\left|\int_X g \circ f^n \cdot h \, \rd m_{\vol} - \int_X g \, \rd\mu \int_X h \, \rd m_{\vol}\right| \le C e^{-\delta_0 n} \|g\|_{C^r} \|h\|_{C^r}.
\end{align}
Moreover, if $\mu$ is non-atomic, it is an ergodic SRB measure, 
which the authors refer to as the \textbf{limit SRB measure}. 
In \cite{Mal25}, Maldonado further established that $(X,f,\mu)$ is Bernoulli, 
confirming a conjecture posed in \cite{BORH23-2}.

The following theorem is the main Bernoulli criterion. 
Note that throughout this paper, 
the Gibbs measure for the topologically mixing Anosov diffeomorphism 
always means the unique equilibrium measure for 
some H\"older continuous potential \cite{Bow75,KH95}. 

\begin{theorem}\label{main theorem-Bernoulli}
Let $f\colon X\to X$ be a topologically mixing $C^{1+\beta}$ 
Anosov diffeomorphism on a closed manifold $X$, and let $\mu$ be a Gibbs measure for $f$. 
Let $\rho_{\bf t}$ be a $C^{1+\beta}$ locally free $\mathbb{R}^k$-action on a closed manifold $Y$, 
and let $\tau: X \to \mathbb{R}^k$ be a $C^{1+\beta}$ cocycle satisfying $\mu(\tau)\neq {\bf 0}$. 
Define the $C^{1+\beta}$ skew product $F\colon X\times Y\to X\times Y$ by
$$
F(x,y):=\left(f(x),\rho_{\tau(x)}(y)\right).
$$
Suppose that either the action $\{\rho_{\bf t}\}$ is exponentially mixing 
with respect to $m$, where $m$ is an SRB measure for $\rho_{\mathbf{t}}$ 
for all $\mathbf{t} \in \mathbb{R}^k$, or the volume measure $m_{\vol}$ is almost exponentially mixing 
and $m$ is the limit SRB measure of $\rho_{\mathbf{t}}$ for any $\mathbf{t} \in \mathbb{R}^k$. 
If the measure $\mu\times m$ is ergodic, 
then $(X\times Y, F,\mu\times m)$ is Bernoulli. 
\end{theorem}

The next theorem is a consequence of results from \cite{DDKN22-1} 
for $\mathbb{R}^k$-actions over hyperbolic systems, 
focusing on the rate of mixing of skew products. For the sake of completeness, we will give a proof in \S \ref{sect:main2}.

\begin{theorem}\label{main theorem-mixing}
Let $f\colon X\to X$ be a topologically mixing $C^{1+\beta}$ Anosov diffeomorphism on a closed manifold $X$, 
and let $\mu$ be a Gibbs measure for $f$. 
Let $\rho_{\bf t}$ be a $C^{1+\beta}$ locally free $\mathbb{R}^k$-action on a closed manifold $Y$ preserving a Borel probability measure $m$, 
and let $\tau: X \to \mathbb{R}^k$ be a $C^{1+\beta}$ cocycle satisfying $\mu(\tau)\neq {\bf 0}$. 
Define the $C^{1+\beta}$ skew product $F\colon X\times Y\to X\times Y$ by
$$
F(x,y):=\left(f(x),\rho_{\tau(x)}(y)\right).
$$
If the action $\{\rho_{\bf t}\}$ is exponentially mixing of all orders with respect to $m$, 
then $(X \times Y, F, \mu\times m)$ is 
exponentially mixing of all orders, and satisfies the Central Limit Theorem. 

Similarly, if $\{\rho_{\bf t}\}$ is merely exponentially mixing with respect to $m$, 
then $(X \times Y, F, \mu\times m)$ is exponentially mixing. 
\end{theorem}

To apply the above results, recall the following standard Weyl chamber action from \cite[Example 8]{DDKN22-2}. 

\begin{example}\label{DDkN22-2-Example 8}
Let $k \geq 1$ and let $\Gamma$ be a cocompact lattice in $\mathrm{SL}(k+1,\mathbb{R})$. 
Let $D_+$ be the group of positive diagonal matrices in $\mathrm{SL}(k+1,\mathbb{R})$. 
Clearly, $D_+$ is isomorphic to $\mathbb{R}^k$. 
The group $D_+$ acts on $\mathrm{SL}(k+1,\mathbb{R})/\Gamma$ by left translation. 
\begin{itemize}
  \item For $k=1$, this one parameter flow $\{\rho_{\bf t}\}$ is the \textbf{geodesic flow}. 
  \item For $k\geq 2$, we obtain a $\mathbb{R}^k$-action $\{\rho_{\bf t}\}$, 
known as the \textbf{Weyl chamber flow}. 
\end{itemize}
Moreover, the action $\{\rho_{\bf t}\}$ is exponentially mixing of all orders (see \cite{BEG20}). 
\end{example}

Theorem \ref{main theorem-Bernoulli} and Theorem \ref{main theorem-mixing} directly yield the following corollary.

\begin{corollary}\label{main SL}
Let $f\colon X\to X$ be a topologically mixing $C^{1+\beta}$ Anosov diffeomorphism on a closed manifold $X$, 
and let $\mu$ be a Gibbs measure for $f$. 
Let $\{\rho_{\bf t}\}$ denote the geodesic flow or a Weyl chamber flow on $Y=\operatorname{SL}(k+1,\mathbb{R})/\Gamma$, 
where $\Gamma\subset\operatorname{SL}(k+1,\mathbb{R})$ is a cocompact lattice, 
and let $m_{\operatorname{Haar}}$ be the Haar measure on $Y$.
Assume that the $C^{1+\beta}$ cocycle $\tau\colon X\to\mathbb{R}^k$ satisfies $\mu(\tau)\neq {\bf 0}$. 
Define the $C^{1+\beta}$ skew product $F\colon X\times Y\to X\times Y$ by
\[
F(x,y):=\left(f(x),\rho_{\tau(x)}(y)\right).
\]
Then the system $(X\times Y,F,\mu\times m_{\operatorname{Haar}})$ is Bernoulli, 
is exponentially mixing of all orders, and satisfies the Central Limit Theorem. 
\end{corollary}
We remark that the above result applies to the measures of maximal entropy for the same skew product system considered here. The proof of the MMEs are of product type will be provided in \cite{DQ-1}. 

In sharp contrast, for the smooth $f$-invariant measure $\mu$ satisfying $\mu(\tau)={\bf 0}$, 
the invariant measure $\mu\times m_{\operatorname{Haar}}$ 
exhibits drastically different statistical behavior, as described below.

\begin{theorem}[\cite{DDKN22-2} Theorem 1.5, Theorem 2.7, Theorem 9.1]\label{DDKN22-2-Theorem 1.5}
Fix an integer $k \geq 1$. 
Let $f :\mathbb{T}^m \to \mathbb{T}^m$ be an 
Anosov diffeomorphism preserving a smooth measure $\mu$. 
For $k=1$, let $\{\rho_{\bf t}\}$ denote the geodesic flow on $\mathrm{SL}(2,\mathbb{R})/\Gamma$, 
and for $k\geq 2$, let $\{\rho_{\bf t}\}$ denote the Weyl chamber flow on 
$\mathrm{SL}(k+1,\mathbb{R})/\Gamma$. 
Let $\tau : \mathbb{T}^m \to \mathbb{R}^k$ 
be a H\"older cocycle satisfying $\mu(\tau)={\bf 0}$ and such 
that $\tau$ is not cohomologous to a cocycle taking value 
in a proper linear subspace of $\mathbb{R}^k$. 
We define the skew product $F_k$ on $\mathbb{T}^m \times \mathrm{SL}(k+1,\mathbb{R})/\Gamma$ by
\[
F_k(x,y) = \big(f(x),\, \rho_{\tau(x)} (y)\big)
\]
Let $m_{\operatorname{Haar}}$ denote the Haar measure 
on the homogeneous space $\mathrm{SL}(k+1,\mathbb{R})/\Gamma$. 
Then the system $(\mathbb{T}^m \times \mathrm{SL}(k+1,\mathbb{R})/\Gamma, 
F_k, \mu \times m_{\operatorname{Haar}})$ is polynomially mixing but not Bernoulli. 
Moreover, this system has the K-property and 
satisfies the Central Limit Theorem. 
\end{theorem}

\begin{remark}[Statistical Phase Transition]\label{remark:Statistical Phase Transition}
The contrast between Corollary \ref{main SL} 
and Theorem \ref{DDKN22-2-Theorem 1.5} reveals a 
statistical phase transition in the mixing rates of $F$. 
While any ergodic measures of maximal entropy  $\mu\times m_{\operatorname{Haar}}$ 
are pushed into the drift regime ($\mu(\tau) \neq {\bf 0}$) 
with exponentially mixing of all orders, 
the system undergoes a sharp transition to polynomially mixing 
at the critical threshold $\mu(\tau) = {\bf 0}$, 
as observed for specific smooth invariant measures in \cite{DDKN22-2}. 
Crucially, this bifurcation is also mirrored in the ergodic-theoretic structure 
of the system: the former system is Bernoulli, whereas the latter is not. 
This implies that the drift plays an important role in the statistical
behavior of the extension.
\end{remark}

\begin{remark}
In fact, all results in this paper remain valid 
when the base map $f$ is an Axiom A diffeomorphism, 
as the required properties continue to hold in this setting. 
For further details on Axiom A diffeomorphisms, 
we refer the reader to \cite{Bow75, PP90}. 
\end{remark}

\subsection{Structure of the paper}

This paper is organized as follows.
In \S \ref{sect: Preliminaries}, we recall the necessary definitions and notation.
In \S \ref{sect:foliations}, we study the invariant foliations and 
Lyapunov exponents of the skew product.
In \S \ref{subsect:Fiberwise constructions and matching}, 
we develop the fiberwise constructions and matching lemmas.
Theorems \ref{main theorem-Bernoulli} and \ref{main theorem-mixing} are proved in
\S \ref{subsect:Proof of Theorem main theorem-Bernoulli} and
\S \ref{sect:main2}, respectively.

\subsection{Acknowledgement}  
This research was supported by National Key R\&D Program of China No. 2024YFA1015100. C. Dong was also grateful for the support by Nankai Zhide Foundation, the Fundamental Research Funds for the Central Universities, Tianjin Natural Science Foundation Outstanding Youth Project 25JCJQJC00310, NSFC 12526207.
Q. Qiao was supported by Nankai Zhide Foundation. 

\section{Preliminaries}\label{sect: Preliminaries}

    In this section, we collect the measure-theoretic definitions 
    and very weak Bernoulli criteria required for 
    the proof of Theorem \ref{main theorem-Bernoulli}.

Recall that a {\bf Lebesgue space} is a Polish space endowed with a Borel probability measure. 
The following well-known result provides a classification of such spaces. 

\begin{lemma}[\cite{Ro52}]\label{DKRH24-Lemma 2.15}
Any two atomless Lebesgue spaces are isomorphic.
\end{lemma} 

Let $(X, \mathcal{B}, \mu)$ be a Borel probability space, 
and let $\mathcal{P} = \{P_1, \dots, P_k\}$ and 
$\mathcal{Q} = \{Q_1, \dots, Q_l\}$ be two finite measurable partitions of $X$. 
The joining partition of $\mathcal{P}$ and $\mathcal{Q}$ is defined by  
$$ 
\mathcal{P} \vee \mathcal{Q} := \{ P_i \cap Q_j : 1 \leq i \leq k, 1 \leq j \leq l \}. 
$$
Let $T: X \to X$ be an automorphism of $(X, \mathcal{B}, \mu)$. 
We say that $\mathcal{P}$ is a {\bf generating partition} 
if $\displaystyle \bigvee_{-\infty}^{\infty}  T^i\mathcal{P}= \mathcal{B}$ 
(up to sets of measure zero). 
Given $\epsilon>0$, a property holds for {\bf $\epsilon$ a.e. 
atom of a partition} $\mathcal{Q}$ if the union
of all atoms that fail to satisfy the property has measure $\leq \epsilon$.

For $A \in \mathcal{B}$, let $\mathcal{P}|_{A}$ denote the induced partition of the space 
$(A, \mu_{|A})$, defined by 
$$
 \mathcal{P}_{|A} := (P_1 \cap A, \dots, P_k \cap A) \quad \text{and} \quad \mu_{|A}(B) = \frac{\mu(A \cap B)}{\mu(A)}. 
 $$ 
On the space of partitions on $(X, \mu)$, 
we introduce the following distance 
$$
\bar{d}(\mathcal{P}, \mathcal{Q}) := \sum_{i=1}^k \mu(P_i \triangle Q_i).
$$
Let $(\mathcal{P}^{(s)})_{s=1}^S$ be a sequence of finite partitions of $(X, \mu)$, 
 where each $\mathcal{P}^{(s)} = (P_1^{(s)}, \dots, P_k^{(s)})$, 
and let $(\mathcal{Q}^{(s)})_{s=1}^S$ be a sequence of finite partitions of $(Y, \nu)$, 
 where each $\mathcal{Q}^{(s)} = (Q_1^{(s)}, \dots, Q_k^{(s)})$.  
If $(X, \mu) = (Y, \nu)$, the distance between these sequences is given by 
$$ 
\bar{d}\left( (\mathcal{P}^{(s)})_{s=1}^S, (\mathcal{Q}^{(s)})_{s=1}^S \right) 
:= \frac{1}{S} \sum_{s=1}^S \bar{d}(\mathcal{P}^{(s)}, \mathcal{Q}^{(s)}). 
$$
More generally, if the sequences are partitions of different spaces, 
we write $(\mathcal{P}^{(s)})_{s=1}^S \sim (\mathcal{Q}^{(s)})_{1}^S$ 
if $\mu(P_i^{(s)}) = \nu(Q_i^{(s)})$ for all $i$ and $s$. 
In this case, the distance is defined as
$$
 \bar{d}\left( (\mathcal{P}^{(s)})_{s=1}^S, (\mathcal{Q}^{(s)})_{s=1}^S \right) 
 = \inf \left\{  
 \bar{d}\left( (\mathcal{P}^{(s)})_{s=1}^S, (\bar{\mathcal{Q}}^{(s)})_{1}^S \right) : \bar{\mathcal{Q}}^{(s)} \sim \mathcal{Q}^{(s)},\, s=1\dots,S \right\},  
 $$
where the infimum is taken over all sequences of partitions $(\bar{\mathcal{Q}}^{(s)})_{s=1}^S$ of the space $(X, \mu)$.

\begin{definition}\label{def:vwB}
Let $T$ be an automorphism of $(X, \mathcal{B}, \mu)$, 
and let $\mathcal{P}$ be a finite partition of $X$. 
We say $\mathcal{P}$ is a {\bf very weak Bernoulli partition 
(vwB partition)} if for every $\epsilon > 0$, 
there exists $N_0 \in \mathbb{N}$ such that for every $N' \ge N \ge N_0$, 
every $S \ge N_0$, and $\epsilon$-a.e. atom $A$ of 
$\displaystyle \bigvee_{i=N}^{N'} T^i \mathcal{P}$, we have
$$ 
\bar{d}\left( (T^{-i}\mathcal{P})_{i=0}^S, (T^{-s}\mathcal{P}_{|A})_{s=0}^S \right) < \epsilon. 
$$ 
\end{definition}

\begin{definition}
An automorphism $T$ of the Borel probability space $(X, \mathcal{B}, \mu)$ 
is said to be {\bf Bernoulli} if it is measure-theoretically isomorphic 
to a Bernoulli shift, i.e., the shift map on 
$(\{1, \dots, k\}^{\mathbb{Z}}, \mathbf{p}^{\mathbb{Z}})$ for some 
$k \in \mathbb{N}$ and a probability vector $\mathbf{p} = (p_1, \dots, p_k)$. 
\end{definition} 

Let $(\mathcal{P}_k)_{k=1}^{+\infty}$ be a sequence of finite measurable 
partitions of $X$. We say that $\mathcal{P}_n$ {\bf converges to the partition 
into points} if, for every measurable set $A \subset X$ and $\epsilon > 0$, 
there exists an $N$ such that for all $n > N$, 
there is a $\mathcal{P}_n$-measurable set $A_n$ satisfying 
$\mu(A \triangle A_n) < \epsilon$. 
According to \cite[Theorems A and B]{OW73}, we have the following result:

\begin{theorem}\label{DKRH24-Theorem 2.18}
Let $(\mathcal{P}_k)_{k=1}^{+\infty}$ be a sequence of finite measurable 
partitions of $(X, \mathcal{B}, \mu)$ that converges to the partition into points. 
If every $\mathcal{P}_n$ is very weak Bernoulli with respect to 
the automorphism $T$, then $(X, \mu, T)$ is a Bernoulli system. 
\end{theorem}

We further assume that $X$ is equipped with a metric $d$ 
that generates the Borel $\sigma$-algebra $\mathcal{B}$. 
For a finite partition $\mathcal{P} = \{P_1, \dots, P_k\}$ of $X$, 
let $\partial \mathcal{P} := \bigcup_{i=1}^k \partial P_i$ denote its boundary. 
We say that $\mathcal{P}$ is a {\bf regular partition} 
if the measure of its $\delta$-neighborhood vanishes as $\delta \to 0$; 
that is, for every $\epsilon > 0$, there exists $\delta > 0$ such that 
$\mu(\mathcal{N}_\delta(\partial \mathcal{P})) < \epsilon$, where 
$$
\mathcal{N}_\delta(\partial \mathcal{P})=\left\{y\in X: d(y,\partial \mathcal{P})<\delta\right\}.
$$

While the definition of a regular partition ensures that the boundary has zero measure, 
our proofs require a stronger partition called $D$-Nice Partition. 

\begin{definition}[\cite{Mal25} $D$-Nice Partition]
A {\bf nice partition} $\mathcal{Q}$ is a finite collection of 
disjoint open sets $\{Q_1, \dots, Q_N\}$ such that for each $i$, 
the boundary $\partial Q_i$ is a finite union of embedded closed disks 
and $\bigcup_i \overline{Q_i} = X$. 

We say that $\mathcal{Q}$ is a {\bf D-Nice partition} if it further satisfies 
the quantitative estimate:
$$ 
\mu(\mathcal{N}_t(\partial \mathcal{Q})) \leq D t 
$$
for some constant $D > 0$ and all sufficiently small $t > 0$. 
\end{definition}

\begin{definition}
Given two probability measure spaces $(X_1, \nu_1), (X_2, \nu_2)$ and $\epsilon>0$, 
a map $\theta : (X_1, \nu_1) \to (X_2, \nu_2)$ is said to be 
\textbf{$\epsilon$-measure preserving} if there exists a measurable set 
$E_1 \subset X_1$ with $\nu_1(E_1) < \epsilon$ 
such that for every measurable set $A \subset X_1 \setminus E_1$, we have
$$
\left| \frac{\nu_2(\theta(A))}{\nu_1(A)} - 1 \right| < \epsilon.
$$
\end{definition}

The following criterion, which originates from \cite[Lemma 2.4]{OW73}
and is formulated as \cite[Corollary 2.21]{DKRH24},
provides a sufficient condition for a partition to be very weak Bernoulli.

\begin{corollary}\label{DKRH24-Corollary 2.21}
  Assume that $X$ is a closed Riemannian manifold
and $\mu$ is nonatomic Borel probability measure. 
Let $T: (X, \mathcal{B}, \mu, d) \to (X, \mathcal{B}, \mu, d)$ 
be an ergodic measure preserving transformation and $\mathcal{P}$ 
a regular partition of $X$.
Suppose that for every $\epsilon > 0$
there exists $N \in \mathbb{N}$ such that
for every $N'\geq N$ and every $S\geq N$,
for $\epsilon$-almost every atom
$\displaystyle A\in\bigvee_{i=N}^{N'}T^i\mathcal P$
there exists an $\epsilon$-measure preserving map 
$\theta = \theta(N, S, A): (A, \mu_{|A}) \to (X, \mu)$ such that for 
$\epsilon$-almost every $x \in (A,\mu_{|A})$,
$$
\frac{1}{S} \operatorname{Card}\left( \left\{ i \in \{0, \dots, S-1\} : d(T^{i}x, T^{i}(\theta (x))) < \epsilon \right\} \right) > 1 - \epsilon, 
$$
then $\mathcal{P}$ is a vwB partition.
\end{corollary}

\section{Invariant foliations}\label{sect:foliations}

This section studies the stable and unstable foliations of $F$
whose projections onto $X$ coincide with the corresponding invariant
foliations of the base map $f$.
Similar constructions were used in \cite{Don} to study smooth time
changes of partially hyperbolic homogeneous flows.

In this section, we only consider an Anosov diffeomorphism $f$ on a closed manifold $X$.  
Let $\{\rho_{\bf t}\}$ be a locally free $C^{1+\beta}$ $\mathbb{R}^k$-action on a closed manifold $Y$. 
Given a $C^{1+\beta}$ function $\tau: X \to\mathbb{R}^k$, 
the skew product $F:X\times Y\to X\times Y$ is defined as 
$$
F(x,y) := \left(fx,\rho_{\tau(x)}y \right).
$$
We let $\pi_X$ and $\pi_Y$
denote the natural projections onto $X$ and $Y$. 
We write $T_{(x,y)}(X\times Y)$ for the tangent space of $X\times Y$
at $(x,y)$ and $T_yY$ for the tangent space of $Y$ at $y$. For a
foliation $\mathcal W$ and a point $z$ on one of its leaves, we denote
by $T_z\mathcal W$ the tangent space to that leaf at $z$.

For any $x \in X$ and $n \in \mathbb{N}$, 
we define the Birkhoff sum $S_n\tau:X\to \mathbb{R}^k$ by 
\begin{align*}
S_n\tau(x):=\sum_{i=0}^{n-1}\tau(f^ix). 
\end{align*}
By convention, we set $S_0\tau=\mathbf0$ and $S_{-n}\tau(x)=-S_n\tau(f^{-n}x)$ for $n\ge 1$.
In particular, the inverse of the forward fiber map at $x$ is 
$$
\rho_{S_n\tau(x)}^{-1}
=\rho_{-S_n\tau(x)}
=\rho_{S_{-n}\tau(f^nx)}.
$$
A straightforward calculation shows that for any $n\in\mathbb{Z}$, the $n$-th iterate of $F$ is given by 
\begin{align}\label{main theorem-Bernoulli-(0)}
F^n(x,y)=(f^nx,\rho_{S_n\tau(x)}y)
\end{align}

For each $x\in X$, let $W^s(x)$ and $W^u(x)$ denote 
the stable and unstable leaves of $f$ passing through $x$. 
For any $x\in X$ and $y\in W^s(x)$, we define 
$$
\beta^s(x,y)=\sum_{k=0}^\infty \left(\tau(f^kx)-\tau(f^ky)\right). 
$$
The sum $\beta^s$ is well-defined since $\tau(f^kx)-\tau(f^ky)$ decays exponentially. 
We define 
\begin{align*}
\widetilde{W}^s(x,z) = \{(y, \rho_{\beta^s(x,y)}z) : y \in W^s(x)\}. 
\end{align*}  

\begin{lemma}\label{Dong23-Lemma 3.1}
For any $(x,z)\in X\times Y$,
$\widetilde{W}^s(x,z)$ is a smooth invariant submanifold of $X\times Y$.
Moreover, it is a stable manifold of $F$ whose projection onto $X$ is $W^s(x)$.
\end{lemma}

\begin{proof}
Fix $x\in X$. For $\mathbf{v} \in E^s(y) = T_y W^s(x)$, we have
$$
\partial_{\mathbf{v}} \beta^s(x,y)
= \sum_{k=0}^\infty \partial_{\mathbf{v}} \left(\tau(f^kx)-\tau(f^ky)\right)
= -\sum_{k=0}^\infty \partial_{Df^k_y \mathbf{v}} \tau(f^ky).
$$
Since $\|Df^k_y \mathbf{v}\|$ decays exponentially,
the map $y \mapsto \beta^s(x,y)$ is smooth on $W^s(x)$.
As $\rho_{t}$ is smooth in the parameter $t\in\mathbb{R}^k$,
the smoothness of $\widetilde{W}^s(x,z)$ is inherited from $W^s(x)$ and $\beta^s$.
Thus $\widetilde{W}^s(x,z)$ is a smooth submanifold.

For any $y\in W^s(x)$ and $z\in Y$, the identity 
$\beta^s(x,y)+\tau(y)=\tau(x)+\beta^s(fx,fy)$ holds.
Thus, we derive 
\begin{align*}
F\left(y,\rho_{\beta^s(x,y)}z\right)
&=\left(fy,\rho_{\beta^s(x,y)+\tau(y)}z\right)
=\left(fy,\rho_{\tau(x)+\beta^s(fx,fy)}z\right)\\
&= \left(fy,\rho_{\beta^s(fx,fy)}\left(\rho_{\tau(x)}z\right)\right)
\in \widetilde{W}^s\left(fx,\rho_{\tau(x)}z\right)
=\widetilde{W}^s\left(F(x,z)\right).
\end{align*}
It follows that $\widetilde{W}^s(x,z)$ is invariant under $F$.

We now show that $(x,z)$ and $(y,\rho_{\beta^s(x,y)}z)$ belong to
the same stable foliation of $F$. 
Fix $x\in X$ and $y\in W^s(x)$.
By a straightforward calculation, for $n\in\mathbb{N}$, we obtain 
\begin{align*}
F^n\left(x,z\right)=\left(f^n(x),\rho_{S_n\tau(x)}z\right),\quad 
F^n\left(y,\rho_{\beta^s(x,y)}z\right)=\left(f^n(y),\rho_{S_n\tau(y)+\beta^s(x,y)}z\right). 
\end{align*}
Since $y\in W^s(x)$ and $\tau$ is $C^{1+\beta}$, the difference 
$$
S_n\tau(x)-\left[S_n\tau(y)+\beta^s(x,y)\right]=-\sum_{k=n}^{\infty}\left(\tau(f^kx)-\tau(f^ky)\right)
$$
decays to $0$ exponentially as $n\to\infty$. 
Thus, $(x,z)$ and $(y,\rho_{\beta^s(x,y)}z)$ 
are in the same stable foliation. 
Hence, $\widetilde{W}^s(x,z)$ is a stable manifold of $F$ whose projection onto $X$ is $W^s(x)$. 
\end{proof}

Similarly, one can construct invariant foliations for each 
$W^{(i)}(x)$ in the filtration 
$$
W^{(1)}(x)\subset W^{(2)}(x)\subset \cdots W^{(s)}(x)=W^s(x),
$$ 
as well as for the unstable foliations. 

Next, we discuss the Lyapunov exponents and pointwise dimensions 
of $f$ along ${W}^s(x)$ with respect to an ergodic measure $\mu$
and those of $F$ along $\widetilde{W}^s(x,z)$ with respect to an ergodic measure $\zeta$ satisfying $\pi_*\zeta=\mu$. 

For the family $\{\rho_{\bf t}\}$, 
we use $D_t \rho_s(z)$ to denote the differential of $\rho_{\bf t}(z)$ 
with respect to the parameter ${\bf t} \in \mathbb{R}^k$ evaluated at $s\in\mathbb{R}^k$, 
and $D_{\omega} \rho_s(z)|_{\omega=z}$ to denote the differential of the diffeomorphism 
$\rho_s$ at the point $z\in Y$. 

\begin{lemma}\label{base LY}
Fix $(x,z)\in X\times Y$. The Lyapunov exponents of $f$ along $W^s(x)$ 
coincide with those of $F$ along $\widetilde{W}^s(x,z)$. 
\end{lemma}

\begin{proof}
 Let $v_0\in T_x W^s(x)$, and 
 suppose the Lyapunov exponent of $f$ along $v_0$ is well-defined. 
By the definition of $\widetilde{W}^s(x,z)$, the tangent vector 
in $T_{(x,z)} \widetilde{W}^s(x,z)$ 
whose projection onto $T_xX$ is $v_0$ is given by 
\begin{align*}
\tilde{v}_0=\left(v_0,  D_y \rho_{\beta^s(x,y)}(z)\big|_{y=x}\,  v_0\right)=
\left(v_0, D_t \rho_{\bf t}(z)\big|_{t=0} \circ D_y \beta^s(x,y)\big|_{y=x} \,v_0\right). 
\end{align*}
For each $n\in\mathbb{N}$, since the projection of
$DF^n_{(x,z)} \tilde{v}_0$ onto $T_{f^n x}X$ is $Df^n_x v_0$, we have 
\begin{align*}
DF^n_{(x,z)}\tilde{v}_0
=\left( Df^n_x v_0,\ D_t \rho_{\bf t}\left( \rho_{S_n\tau(x)}z \right)\big|_{t=0}
\circ D_{y'} \beta^s(f^n x, y')\big|_{y'=f^n x}\, Df^n_x v_0 \right).
\end{align*}
Since $D_t \rho_{\bf t}\big(z\big)\big|_{t=0}$ is uniformly bounded over $z\in Y$ 
and $D_{y'} \beta^s(f^n x, y')\big|_{y'=f^n x}$ decays exponentially to $0$ as $n\to\infty$, 
the vertical component of the tangent vector remains bounded relative to 
its base component. We conclude that 
$$
\lim_{n\to\infty} \frac{1}{n} \log \| Df^n_x v_0 \|
=\lim_{n\to\infty} \frac{1}{n} \log \left\|
DF^n_{(x,z)} \tilde{v}_0\right\|.
$$
Therefore, the Lyapunov exponents along $\widetilde{W}^s(x,z)$ for $F$
coincide with the Lyapunov exponents along $W^s(x)$ for $f$.
\end{proof}

The case for the unstable manifold is analogous. 
We omit the proof here, but present the expression below. 
For any $x,y\in X$, we define 
\begin{align*}
\beta^u(x,y):=-\sum_{k=1}^\infty \left( \tau(f^{-k}x)-\tau(f^{-k}y) \right). 
\end{align*}
The unstable foliation corresponding to $W^u(x)$ is given by 
\begin{align*}
\widetilde{W}^u_f(x,z) = \{(y, \rho_{\beta^u(x,y)}z) : y \in W^u(x)\}. 
\end{align*}

We next identify the Lyapunov exponents for an ergodic invariant product measure. 

\begin{lemma}\label{all LY}
Suppose $\zeta = \mu \times \nu$ is an ergodic $F$-invariant measure on $X \times Y$. 
Then the Lyapunov exponents of $F$ with respect to $\zeta$ are the 
Lyapunov exponents of $f$ with respect to $\mu$ on $X$ 
and the Lyapunov exponents of $\rho_{\mu(\tau)}$ with respect to $\nu$ on $Y$.

\end{lemma}

\begin{proof}
A direct computation shows that the differential of $F$ at $(x,y)$ is
\begin{align*}
D_{(x,y)}F = \begin{pmatrix} D_x f & 0 \\ * & D_y \rho_{\tau(x)} \end{pmatrix},
\end{align*}
where $D_x f$ denotes the differential of $f$ at $x$, and $D_y \rho_{\tau(x)}$ denotes the differential of $\rho_{\tau(x)}$ at $y$. Iterating this formula $n$ times, we obtain
\begin{align*}
D_{(x,y)}F^n = \begin{pmatrix} D_x f^n & 0 \\ * & D_y \rho_{S_n\tau(x)} \end{pmatrix}.
\end{align*}
We define 
$$
\Upsilon := \sup\left\{ \|D_y \rho_{\mathbf{t}}\| : y \in Y,\ \mathbf{t} \in \mathbb{R}^k,\ \|\mathbf{t}\| \leq 1 \right\} + 1 < \infty,
$$
where the finiteness follows from the compactness of $Y$ and the continuity of $\rho_{\mathbf{t}}$ in $\mathbf{t}$ and $y$.

Fix $\epsilon > 0$ and a tangent vector $v \in T_y Y$, 
identifying it with $(0, v) \in T_{(x,y)}(X \times Y)$. 
By the Birkhoff Ergodic Theorem, for $\mu$-almost every $x \in X$, 
there exists $N = N(x,\epsilon) > 0$ such that 
$\|S_n\tau(x) - n\mu(\tau)\| < n\epsilon$ for all $n \geq N$. 
For $\zeta$-almost every $(x,y) \in X \times Y$, we have
\begin{align*}
\frac{1}{n}\log \left\| D_y \rho_{S_n\tau(x)} v \right\|
&= \frac{1}{n}\log \left\| D_y \rho_{n\mu(\tau)} \circ D_y \rho_{S_n\tau(x) - n\mu(\tau)} v \right\| \\
&\leq \frac{1}{n}\log \left\| D_y \rho_{n\mu(\tau)} v \right\| + \frac{1}{n}\log \left\| D_y \rho_{S_n\tau(x) - n\mu(\tau)} \right\| \\
&\leq \frac{1}{n}\log \left\| D_y \rho_{n\mu(\tau)} v \right\| + \frac{\log \Upsilon^{\|S_n\tau(x) - n\mu(\tau)\|}}{n} \\
&\leq \frac{1}{n}\log \left\| D_y \rho_{n\mu(\tau)} v \right\| + \epsilon \log \Upsilon.
\end{align*}
The reverse inequality can be established similarly. 
Letting $n \to \infty$ and then $\epsilon \to 0$, we conclude that
\begin{align}\label{all LY-(1)}
\lim_{n\to \infty} \frac{1}{n}\log \left\| D_y \rho_{S_n\tau(x)} v \right\|
= \lim_{n\to \infty}\frac{1}{n}\log \left\| D_y \rho_{\mu(\tau)}^n v \right\|.
\end{align}
By the Multiplicative Ergodic Theorem \cite{Oseledec}, 
for $\nu$-almost every $y \in Y$ 
(hence for $\zeta$-almost every $(x,y) \in X \times Y$), 
the limit in \eqref{all LY-(1)} exists and equals the Lyapunov exponent of $\rho_{\mu(\tau)}$ 
with respect to $\nu$ along the tangent vector $v \in T_y Y$. 
Combining this with Lemma \ref{base LY}, we complete the proof.
\end{proof}

For convenience, we use $\widetilde{W}^u$ to denote the unstable manifold 
whose projection is the unstable manifold of $f$, 
and $\widehat{W}^u$ to denote the unstable manifold 
corresponding to the tangent space $\{0\}\times T_yY$. 
Similar definitions and notations are adopted for the stable case.

\section{Bernoulli property}\label{sect: Bernoulli property}

Here we endow the product space $X\times Y$ with the product metric 
$$
d_{X\times Y}((x_1,y_1),(x_2,y_2))=\max\left\{ d_X(x_1,x_2), d_Y(y_1,y_2)\right\}, 
$$
where $d_X$ and $d_Y$ denote the metrics on $X$ and $Y$. 
For brevity, we often omit the subscripts and denote the metric by $d$ 
whenever the context is clear. 
For convenience, in both cases of Theorem \ref{main theorem-Bernoulli}, 
we use $\delta_0$ to denote the exponential decay rate 
such that \eqref{Mal25-(1.2)} or \eqref{DDKN22-1-(3.1)} holds. 

\subsection{Fiberwise constructions and matching}\label{subsect:Fiberwise constructions and matching}

This subsection adapts the arguments of \cite{DKRH24, Mal25}
to the present skew-product setting. 
The main difference is that the fiber dynamics depend on the base point 
through the cocycle sums $S_n\tau(x)$. 
Rather than restating all the preliminary definitions, 
we focus on the modifications needed to obtain the required uniformity in this setting.

Since $\mu$ is ergodic, the Birkhoff ergodic theorem implies that for $\mu$-a.e. $x\in X$, 
\begin{align}\label{main theorem-Bernoulli-(1)}
\lim_{n\to + \infty} \frac{S_n\tau(x)}{n}
=\lim_{n\to - \infty} \frac{S_n\tau(x)}{n}
=\int_X \tau\,\rd\mu=\mu(\tau). 
\end{align} 
Thus, by Egorov's theorem and the regularity of $\mu$, for any $\delta>0$, 
there exists a compact subset $X_\delta\subset X$ with
$\mu(X_\delta)>1-\delta$ and $N_\delta\in\mathbb N$ such that
$$
\|S_n\tau(x)-n\mu(\tau)\|<|n|\delta
$$
for every $x\in X_\delta$ and every $n\in\mathbb Z$ with $|n|\ge N_\delta$. 

We refer the reader to \cite{DKRH24, Mal25} for the Pesin theory concerning 
unstable and fake center-stable foliations. 
As established in Lemma \ref{all LY}, 
the invariant foliations of $F$ coincide with the stable and 
unstable foliations of the diffeomorphism $\rho_{\mu(\tau)}$. 
Notably, as observed in Lemma \ref{Dong23-Lemma 3.1}, 
these foliations differ from those of $f$; 
in the latter case, the $Y$-coordinate varies along the leaves 
rather than remaining constant. 
Consequently, for sufficiently large $|n|$, the fiber dynamics of $F^n$ 
over $X_\delta$ can be shadowed by the iterates of $\rho_{\mu(\tau)}$. 
This ensures that the Pesin-theoretic constructions in \cite{DKRH24, Mal25} 
remain valid in our fiberwise dynamics. 

We use the parameter $\iota$ to replace $\tau$ from \cite{Mal25}, 
since $\tau$ represents the cocycle in this paper. 
Following \cite[\S 2.2]{Mal25}, for a given compact set $X_{\delta}$, 
we introduce the subset $\mathcal{O}_{\iota,\delta} \subset Y$ 
such that the product $X_{\delta} \times \mathcal{O}_{\iota,\delta}$ 
constitutes the Pesin set of $F$ on $X \times Y$, 
which corresponds to $\mathcal{P}_\tau$ in \cite{Mal25}. 
Furthermore, the previous discussion ensures that $\mathcal{O}_{\iota,\delta}$ 
exhibits uniform asymptotic behavior across all $x \in X_{\delta}$. 
As $\delta\to 0, \iota\to 0$, 
$(\mu\times m)(X_\delta \times \mathcal{O}_{\iota,\delta})\to 1$. 
For convenience, given $X_{\delta}$, 
when we consider a subset on $Y$ whose measure is close to $1$, 
we always assume that this subset is contained in $\mathcal{O}_{\iota,\delta}$ for 
some sufficiently small $\iota$. 

Specifically, the construction of ``good boxes'' on $Y$ becomes 
essentially independent of the $X$-coordinate when restricted to $X_{\delta}$, 
as the fiberwise action of the skew product is controlled 
by the uniform Birkhoff average. 
This justifies the following generalization of \cite[Lemma 3.1]{Mal25} 
to our skew product setting: 

\begin{lemma} \label{Mal25-Lemma 3.1}
For any $\delta > 0$, 
there exist a compact subset $X_\delta \subset X$ with $\mu(X_{\delta}) > 1 - \delta$ 
and a collection of good boxes covering a subset of $Y$ of large measure, 
such that the conclusions of \cite[Lemma 3.1]{Mal25} hold for all points 
in $X_\delta \times Y$ whenever $n$ is sufficiently large. 

Notably, the construction of these good boxes on the fibers 
and all associated estimates are uniform across $X_\delta$, 
depending only on the integrated rotation vector $\mu(\tau)$. 
\end{lemma}

We follow the notation of \cite{Mal25}, such as the Lyapunov coordinate
map $h_x$ introduced in \cite[Lemma 2.8]{Mal25}. 
To avoid repeating the preliminary constructions, 
we refer to \cite{Mal25} for their detailed definitions. 
Throughout this section, these notations are used only for the fiber dynamics on $Y$; 
in particular, $\mathbb{R}^{cs}$ denotes the Euclidean space associated with
$\rho_{\mu(\tau)}$ and is independent of the base dynamics $f$.

In the arguments below, the iterates $f^n$ appearing in \cite{Mal25} are replaced, 
for a fixed base point $x$, by the fiber maps $\rho_{S_n\tau(x)}$. 
The uniform estimates established above on the chosen compact subset 
of the base allow this replacement to be made uniformly in $x$. 
We therefore focus only on the points in the arguments of \cite{Mal25} 
where this replacement requires a modification.

\begin{proposition}\label{Mal25-Proposition 3.5}
Suppose the action $\{\rho_{\bf t}\}$ is exponentially mixing 
with respect to the nonatomic measure $m$. 
Fix small $ \epsilon,\epsilon_0>0 $, and 
assume that the conditions of \cite[Proposition 3.5]{Mal25} hold. 
Then, there exist $ \gamma_0>0$ and $n_0 \in\mathbb{N}$ such that if 
 $(G_j)_{j \in J(n)}$ and $(\overline{C}_i)_{i \in I(n)}$ are good boxes 
 from Lemma \ref{Mal25-Lemma 3.1} 
 associated with value $\gamma_0$ and the corresponding parameters, 
 then for all $n \geq n_0$, $i \in I(n)$, $j \in J(n)$, and 
 any $x \in X_{\epsilon_0}$, 
\begin{align*}
(1+\epsilon)^{-1} \leq \frac{m\left(C_j \cap \rho^{-1}_{S_{n}\tau(x)} \overline{C}_i \right)}{m(C_j) m(\overline{C}_i)} \leq 1+\epsilon. 
\end{align*}
\end{proposition}

\begin{proof}
The proof follows a similar logic to that of \cite[Proposition 3.5]{Mal25}, 
with several necessary modifications. 
Following the approach in \cite[Lemma 3.1]{Mal25}, 
we use Lipschitz functions $h$ and $g$ to 
approximate the indicator functions of $C_j$ and $\overline{C}_i$, 
respectively. Specifically, the estimate in \cite[(3.26)]{Mal25} is replaced by:
\begin{align}\label{Mal25-(2.26)}
\frac{m(C \cap \rho^{-1}_{S_{n}\tau(x)}\overline{C})}{m(C) m(\overline{C})}
= \frac{m(C \cap \rho^{-1}_{S_n\tau(x)}\overline{C})}{m(h \cdot (g \circ \rho_{S_n\tau(x)}))}
\cdot \frac{m(h \cdot (g \circ \rho_{S_n\tau(x)}))}{m(h)m(g)}
\cdot \frac{m(h)}{m(C)}
\cdot \frac{m(g)}{m(\overline{C})}.
\end{align}
Thus, by applying \eqref{DDKN22-1-(3.1)} instead of the exponential mixing property of $f$, 
the rest of the argument remains applicable to our setting, 
which completes the proof. 
\end{proof}

By an analogous argument, we obtain the following result for the limit SRB measure. 
Further comments can be found in \cite{Mal25}.  

\begin{proposition}\label{Mal25-Proposition 3.6}
Suppose the volume measure $m_{\vol}$ is almost exponentially mixing and 
$m$ is the limit SRB measure of $\rho_{\mathbf{t}}$ 
for any $\mathbf{t} \in \mathbb{R}^k$. 
Fix small $ \epsilon,\epsilon_0>0 $, and 
assume that the conditions of \cite[Proposition 3.6]{Mal25} hold. 
Then, there exist $ \gamma_0>0$ and $n_0 \in\mathbb{N}$ such that if 
 $(G_j)_{j \in J(n)}$ and $(\overline{C}_i)_{i \in I(n)}$ are good boxes 
 from Lemma \ref{Mal25-Lemma 3.1} 
 associated with value $\gamma_0$ and the corresponding parameters, 
 then for all $n \geq n_0$, $i \in I(n)$, $j \in J(n)$, and any $x \in X_{\epsilon_0}$, 
\begin{align*}
(1+\epsilon)^{-1} \leq \frac{m_{\vol}\left(C_j \cap \rho^{-1}_{S_{n}\tau(x)} \overline{C}_i \right)}{m_{\vol}(C_j) m(\overline{C}_i)} \leq 1+\epsilon. 
\end{align*}
\end{proposition}

\begin{proof}
The proof follows the same lines as that of Proposition \ref{Mal25-Proposition 3.5}, 
except for replacing \eqref{Mal25-(2.26)} with the following equality: 
\begin{align*}
\frac{m_{\vol}(C \cap \rho^{-1}_{S_n\tau(x)}\overline{C})}
      {m_{\vol}(C)m(\overline{C})}
=\frac{m_{\vol}(C \cap \rho^{-1}_{S_n\tau(x)}\overline{C})}
     {m_{\vol}(h\cdot(g\circ\rho_{S_n\tau(x)}))}
\cdot \frac{m_{\vol}(h\cdot(g\circ\rho_{S_n\tau(x)}))}{m_{\vol}(h)m(g)}
\cdot\frac{m_{\vol}(h)}{m_{\vol}(C)}
\cdot\frac{m(g)}{m(\overline{C})},
\end{align*}
and applying \eqref{Mal25-(1.2)} in place of \eqref{DDKN22-1-(3.1)}. 
\end{proof} 

Next, we describe the construction of fake center-stable foliations, 
following the approach in \cite[Proposition 4.1]{Mal25}. 

\begin{proposition}\label{Mal25-Proposition 4.1}
Suppose $m$ is an ergodic SRB measure of $\rho_{\mathbf{t}}$ 
for all $\mathbf{t} \in \mathbb{R}^k$. 
Fix a small $ \epsilon>0 $, $\epsilon_0=\frac{1}{100}\epsilon$ and $\delta=\epsilon^{{100}^{100}}$, 
and assume that the conditions of \cite[Proposition 4.1]{Mal25} hold. 
For sufficiently large $n$, let $(G_j)_{j \in J(n)}$ and 
$(\overline{C}_i)_{i \in I(n)}$ denote the good boxes from Lemma \ref{Mal25-Lemma 3.1}. 
For $i\in I_n$, let $\overline{\mathcal{F}}_i^{cs}$ be the foliation 
of $\overline{C}_i$ by vertical lines $h_{z_i} (\mathbb{R}^{cs})$. 

Under these assumptions, there exist $n_0\in\mathbb{N}$ such that 
for $n > n_0$, there exist subsets $J_n' \subset J (n)$ and $I_n(j) \subset I(n)$ 
(for each $j \in J(n)$) satisfying
$$
m\left( \bigcup_{j \in J_n'} C_j \right) > 1 - \epsilon^{80} \text{ and } 
m\left( \bigcup_{i \in I_n(j)} \overline{C}_i \right) > 1 - \epsilon^3
$$
Furthermore, for each $j \in J_n$ and $x\in X_{\epsilon_0}$, 
we decompose the following intersection into its connected components: 
$$
C_j \cap \rho^{-1}_{S_n\tau(x)}\left( \bigcup_{i \in I_n(j)} \overline{C}_i \right) = \bigsqcup_{l \in L(j,x)} G_l. 
$$
In this decomposition, 
there exists a subcollection $L'(j,x) \subset L(j,x)$ with 
$$
m_{C_j}\left( \bigcup_{l \in L'(j,x)} G_l \right) \geq 1 - \epsilon^2
$$ 
such that for each $\displaystyle y \in \bigsqcup_{l \in L'(j,x)} G_l$, 
the foliation defined by
$$
\mathcal{F}_j^{cs,n}((x,y)): = F^{-n}\left( \overline{\mathcal{F}}_i^{cs}(F^n (x,y)) \right) \cap G_l,
$$
satisfies properties (F1) and (F2) of \cite[Proposition 4.1]{Mal25}. 
Finally, property (F3) is replaced by the following:
\begin{itemize}
  \item[(F3)] For any $x \in X_{\epsilon_0}$, 
if $(x,y), (x,y') \in \mathcal{F}_j^{cs,n}((x,w))$ for some $w\in Y$, 
  then for any $0 \leq k \leq n$, we have 
  $$
d(F^k(x,y),F^k(x,y'))=d(\rho_{S_{k}\tau(x)} y, \rho_{S_{k}\tau(x)} y') \leq \iota^{-2} e^{2\sqrt{\delta} k} d(y, y').
  $$
\end{itemize} 

We call such a foliation an $(\epsilon,n)$-fake cs foliation.  
\end{proposition}
 
\begin{proof}
Since the limit in \eqref{main theorem-Bernoulli-(1)} is uniform 
over the compact subset $X_{\epsilon_0}$, 
the same argument from \cite[Proposition 4.1]{Mal25} is applicable to our case. 
We omit the detailed proof here for simplicity. 
\end{proof}

\begin{remark}\label{remark-Fcsn}
By definition, the foliation $\mathcal{F}_i^{cs}$ is independent of the 
$X$-coordinate, while $\mathcal{F}_i^{cs,n}$ depends on the $X$-coordinate. 
Moreover, since $\pi(\mathcal{F}_i^{cs,n}(x,y))=\{x\}$, 
it can be seen as a foliation on $Y$. 
For convenience, we use $\mathcal{F}^{cs,n}_j(x,\cdot)$ 
to denote the family of foliations for a fixed $x$. 
\end{remark}

The following auxiliary lemma establishes the equidistribution 
along thin fake center stable saturated subsets of exponentially small boxes. 
The proof is analogous to that of \cite[Lemma 5.2]{Mal25}, 
following the same argument used above: 
we employ a compact subset $X_{\epsilon_0}$ to obtain the uniform asymptotic behavior of 
$S_n\tau(x)$ for $x\in X_{\epsilon_0}$ and then apply \eqref{main theorem-Bernoulli-(0)} and \eqref{main theorem-Bernoulli-(1)}. 
Here, $\lfloor a \rfloor$ denotes the integer part of $a$. 
Given two sequences $a_n,b_n$, the notation $a_n \ll b_n$ means that 
$\displaystyle \lim_{n\to \infty}\frac{a_n}{b_n}=0$. 
We denote the dimension of the manifold $Y$ by $m_Y$. 

The next lemma is a well-known result, 
and its proof is almost identical to \cite[Proposition 20.2.6]{KH95}; 
thus, we omit its proof. 

 
\begin{lemma}\label{Bowen ball-sum}
Let $f\colon X\to X$ be a topologically mixing 
Anosov diffeomorphism on a closed manifold $X$. 
Given a Lipschitz continuous function $\tau:X\to \mathbb{R}^k$, 
there exist $r_0>0$ and $C_\tau>0$ such that, 
for every $h\in \mathbb{N}$, every $r \in (0, r_0]$, 
and every $x_1,x_2\in X$ satisfying
$$  
d_X(f^j x_1,f^j x_2)<r,\quad 0\leq j\leq h,  
$$
we have 
\begin{align*}
  \max_{0\leq j\leq h}\|S_j\tau(x_1)-S_j\tau(x_2)\|\leq C_\tau r.
\end{align*} 
\end{lemma}

Since $\{\rho_{\bf t}\}$ is a smooth action on a compact manifold, 
there exsits $C_\rho>0$ such that
\begin{align}\label{gruop action bound}
d_Y(\rho_{\mathbf u}z,\rho_{\mathbf v}z)
\leq C_\rho\|\mathbf u-\mathbf v\|,
\quad \forall\, z\in Y,\quad \forall\,\mathbf u,\mathbf v\in\mathbb R^k.
\end{align}
For each sufficiently small $\epsilon>0$, we choose $r_b=r_b(\epsilon)$ satisfying 
\begin{align}\label{eq:B2-base-radius}
0<r_b<\min\left\{r_0,\epsilon,
\frac{\epsilon}{4C_\rho C_\tau}\right\}.
\end{align}

\begin{lemma}\label{Mal25-Lemma 5.2}
Suppose either that the action $\{\rho_{\bf t}\}$ is exponentially mixing 
with respect to $m$, where $m$ is an SRB measure of $\rho_{\mathbf{t}}$ 
for all $\mathbf{t} \in \mathbb{R}^k$; or that 
the volume measure $m_{\vol}$ is almost exponentially mixing and 
$m$ is the limit SRB measure of $\rho_{\mathbf{t}}$ 
for any $\mathbf{t} \in \mathbb{R}^k$. 

Fix any $\epsilon_0>0$.  Let $\epsilon > 0$ and $(\xi_n)_{n \in \mathbb{N}}$ be a sequence of positive numbers satisfying
\begin{align*}
e^{-2\frac{\delta_0 \|\mu(\tau)\|}{4m_{Y}+2}n} \ll \xi_n \ll e^{-(\delta+4\sqrt{\delta})n},
\end{align*}
where $\delta = \epsilon^{100^{100}}$ and $\delta_0$ is the parameter 
in \eqref{DDKN22-1-(3.1)} or \eqref{Mal25-(1.2)}. 
We define $m_n:=\lfloor\epsilon n\rfloor$ and
$$
G_n:=X_{\epsilon_0}\cap f^{-m_n}X_{\epsilon_0}.
$$
It follows that $\mu(G_n)>1-2\epsilon_0$. 
Then, there exists $n_0 = n_0(\xi,\epsilon,\epsilon_0) \in\mathbb{N}$
such that for each $n > n_0$ and $x\in G_n$,
there exists a subset $\mathcal{K}^x_n \subset Y$, 
a measurable $u$-subordinate partition $\{W_n((x,y))\}_{y \in T_{n,x}}$ of 
$\mathcal{K}^x_n$, and a finite family of pairwise disjoint subsets 
$\{B^x_j\}_{j \in J(n,x)}$ of $Y$ satisfying:
\begin{itemize}
\item[(U1)] $m(\mathcal{K}^x_n) > 1 - \epsilon$;
\item[(U2)] For each $y \in T_{n,x}$, $\{x\}\times W_n((x,y))$ 
is a piece of unstable manifold $\widehat{W}^u((x,y))$ with its size 
between $(1+\epsilon)^{-1} \xi_{\lfloor \epsilon n \rfloor}$ and 
$(1+\epsilon) \xi_{\lfloor \epsilon n \rfloor}$; 
\item[(B1)] For every pair $x_1,x_2\in G_n$ satisfying
\begin{align}\label{eq:B2-base-condition}
d_X(f^jx_1,f^jx_2)<r_b,\quad 0\leq j\leq n.
\end{align}
and for each $y\in T_{n,x_2}$, there exists a subset 
$J'_{n;x_2,x_1}(y)\subset J(n,x_1)$ such that
$$
m\left(\bigcup_{i\in J'_{n;x_2,x_1}(y)}B_i^{x_1}\right)>1-\epsilon
$$
and, for every $i\in J'_{n;x_2,x_1}(y)$,
\begin{align*}
m_{W_n((x_2,y))}
\left(\rho_{S_{m_n}\tau(x_2)}^{-1}B_i^{x_1}\right)
\geq\frac{m(B_i^{x_1})}{1+\epsilon}.
\end{align*}
Such index sets may depend on both base points and on the plaque. 
In particular, the case $x_1=x_2$ is also included.

\item[(B2)] 
For every pair $x_1,x_2\in G_n$ satisfying
\eqref{eq:B2-base-condition}, each $j\in J(n,x_1)$, any
$y_1,y_2\in B_j^{x_1}$, and any integer $k$ with $m_n\leq k\leq n$, one has 
\begin{align*}
d\left(F^{k-m_n}(f^{m_n}x_1,y_1),F^{k-m_n}(f^{m_n}x_2,y_2)\right)
<d_X(f^kx_1,f^kx_2)+\epsilon. 
\end{align*}
\end{itemize}
\end{lemma}

\begin{proof}
The primary difference between this lemma and \cite[Lemma 5.2]{Mal25} is that
$\mathcal{F}_i^{cs,n}((x,y))$ depends on $x$ (see Remark \ref{remark-Fcsn}).
The only steps in the proof that require clarification are properties (B1) and (B2);
once these properties are established, the remaining parts follow straightforwardly
from the arguments in \cite[Lemma 5.2]{Mal25}.

Choose $0<\eta<\epsilon^2/10$ such that
$$
(1+\eta)^{26}<1+\epsilon,\quad
(1+\eta)^{-5}\left((1+4\eta)^{-2}-\eta\right)>1-\epsilon.
$$
Throughout the proof, the parameters $\epsilon,\epsilon_0,\delta$ and the observation time
$m_n=\lfloor\epsilon n\rfloor$ remain as in the statement.
We increase $n_0$ when necessary.

For each $x\in G_n$, we follow the construction of
\cite[Lemma 5.1]{Mal25} at time $m_n$, as in the proof of
\cite[Lemma 5.2]{Mal25}, replacing $f^{m_n}$ by
$\rho_{S_{m_n}\tau(x)}$.
The estimates in \cite[(5.10) and (5.18)]{Mal25}
are replaced by Proposition \ref{Mal25-Proposition 3.5}
and Proposition \ref{Mal25-Proposition 3.6}, respectively.
This gives $\mathcal K_n^x$ and its measurable $u$-subordinate
partition $\{W_n((x,y))\}_{y\in T_{n,x}}$, satisfying (U2) and
$$
m(\mathcal K_n^x)>1-\eta>1-\epsilon.
$$
Thus (U1) holds. We also retain the estimates (U3) and (E1)
of the cited lemma, with the bad set chosen as at the beginning
of the proof of \cite[Lemma 5.2]{Mal25}.
These constructions may depend on $x$, but the estimates are uniform
for $x\in G_n$.

Since $x\in G_n$, we have $f^{m_n}x\in X_{\epsilon_0}$.
For each eligible coarse box $\overline C_i$, choose
$$
\overline y_i\in\overline C_i\cap\mathcal O_\iota.
$$
By Proposition \ref{Mal25-Proposition 4.1},
the $(\epsilon,n)$-fake center-stable foliation
$\mathcal F_i^{cs,n}(f^{m_n}x,\cdot)$ is constructed
on the corresponding good components.

Following the construction in \cite[(5.25)]{Mal25},
choose a collection of pairwise disjoint sets
$\{\widehat B_{i,j}\}_{j\in J_n(i)}$ in
$$
\widehat W^u((f^{m_n}x,\overline y_i))
\cap\mathcal F_i^{cs,n}(f^{m_n}x,\cdot),
$$
all having the same unstable leaf volume and diameter at most $\epsilon^2 M_0^{-n}$, where 
\begin{align*}
M_0:=
\max\left\{
1,\sup_{(x,y)\in X\times Y}\|D_{(x,y)}F\|
\right\},
\end{align*}
such that the analogue of \cite[(5.25)]{Mal25} holds.
Define
$$
B_{i,j}^x
=
\bigcup_{z\in\widehat B_{i,j}}
\mathcal F_i^{cs,n}((f^{m_n}x,z))\cap\widetilde T_i.
$$
Here $\widetilde T_i$ is the set defined in
\cite[Lemma 3.1 (P7)]{Mal25}.
Clearly, the sets $B_{i,j}^x$ depend on $x$.
Since $\pi(\mathcal F_i^{cs,n}((f^{m_n}x,z)))=\{f^{m_n}x\}$,
we identify each $B_{i,j}^x$ with a subset of $Y$.
Collecting these sets over all $i$ gives
the desired family $\{B_j^x\}_{j\in J(n,x)}$, 
as claimed in the first
part of property (B1).

Next, we turn to the proof of the second part of property (B1).

Crucially, because each $\overline{C}_i$ is constructed independently
of $x\in X_{\epsilon_0}$, the equidistribution of the unstable manifold
$\widehat{W}^u$ along the exponentially small boxes $\overline{C}_i$
holds uniformly for all $x\in X_{\epsilon_0}$.
This uniformity ensures that the approach of \cite[Lemma 5.1 (E1)]{Mal25}
remains applicable when the inverse iterates $f^{-n}$ are replaced by
$\rho_{S_n\tau(x)}^{-1}$.

To prove the second part of (B1), fix $x_1,x_2$ satisfying
\eqref{eq:B2-base-condition} and let $y\in T_{n,x_2}$.
Following the proof of \cite[Lemma 5.2 (B1)]{Mal25},
we decompose
$$
W_n((x_2,y))
\cap
\rho_{S_{m_n}\tau(x_2)}^{-1}B_i^{x_1}
$$
into its connected components. 
The corresponding estimate (E1) is obtained
from \cite[Lemma 5.1]{Mal25}, with
Proposition \ref{Mal25-Proposition 3.5} in the exponentially
mixing case and Proposition \ref{Mal25-Proposition 3.6} in the
almost exponentially mixing volume case.
The leaf-volume comparison is supplied by
\cite[Lemma 2.22]{Mal25}, while the required estimates for the
good boxes and their conditional measures follow from
Lemma \ref{Mal25-Lemma 3.1}, especially (P8).
Finally, the blocks $B_i^{x_1}$ were constructed above
from the fake center-stable foliations on $\mathcal O_\iota$
provided by Proposition \ref{Mal25-Proposition 4.1}.
All these estimates are uniform for the base points under
consideration.

Therefore, the argument of \cite[(5.28)--(5.32)]{Mal25},
with $W_n(x)$, $f^{-m_n}$ and $B_i$ replaced respectively by
$W_n((x_2,y))$, $\rho_{S_{m_n}\tau(x_2)}^{-1}$ and $B_i^{x_1}$,
gives a subset
$J'_{n;x_2,x_1}(y)\subset J(n,x_1)$ such that
$$
m\left(
\bigcup_{i\in J'_{n;x_2,x_1}(y)}
B_i^{x_1}
\right)>1-\epsilon
$$
and, for every $i\in J'_{n;x_2,x_1}(y)$,
$$
m_{W_n((x_2,y))}
\left(
\rho_{S_{m_n}\tau(x_2)}^{-1}B_i^{x_1}
\right)
\geq
\frac{m(B_i^{x_1})}{1+\epsilon}.
$$
Therefore, we complete the proof of property (B1).

To verify the property (B2), 
fix $x_1,x_2\in G_n$ satisfying \eqref{eq:B2-base-condition} 
and $y_1,y_2\in B_j^{x_1}$ for some $j\in J(n,x_1)$.
Before relabelling the blocks, we write $B_j^{x_1}=B_{i,\ell}^{x_1}$.
By the construction of $B_j^{x_1}$, there exist
$z_1,z_2\in\widehat B_{i,\ell}$ for some $\widehat B_{i,\ell} $ such that
$$
(f^{m_n}x_1,z_1)\in\mathcal F_i^{cs,n}((f^{m_n}x_1,y_1)),
\quad
(f^{m_n}x_1,z_2)\in\mathcal F_i^{cs,n}((f^{m_n}x_1,y_2)).
$$
The diameter bound of $\widehat B_{i,\ell}$ yields 
$d_Y(z_1,z_2)\leq\epsilon^2M_0^{-n}$, 
while Lemma \ref{Mal25-Lemma 3.1} gives 
$$
d_Y(y_1,z_1)\leq\overline r_n,\quad d_Y(y_2,z_2)\leq\overline r_n,
$$
where the parameter $\overline{r_n} \ll e^{-3\sqrt{\delta}n}$ is introduced 
in the assumptions of \cite[Lemma 3.1]{Mal25} and is applied here in Lemma \ref{Mal25-Lemma 3.1}. 

For $m_n\leq k\leq n$, we set $h=k-m_n$.
By \eqref{eq:B2-base-condition}, for $0\leq q\leq h$, we have 
$$
d_X(f^{m_n+q}x_1,f^{m_n+q}x_2)<r_b.
$$
It follows from Lemma \ref{Bowen ball-sum} that 
$$
\|S_h\tau(f^{m_n}x_1)-S_h\tau(f^{m_n}x_2)\| \leq C_\tau r_b.
$$
Combining Proposition \ref{Mal25-Proposition 4.1} (F3), 
the diameter bound of $\widehat B_{i,\ell}$,
and \eqref{gruop action bound}, we obtain 
\begin{align*}
&d\left(F^h(f^{m_n}x_1,y_1),F^h(f^{m_n}x_2,y_2)\right)\\
\leq & d\left(F^h(f^{m_n}x_1,y_1),F^h(f^{m_n}x_1,z_1)\right)
+d\left(F^h(f^{m_n}x_1,z_1),F^h(f^{m_n}x_1,z_2)\right)\\
&\quad\quad+d\left(F^h(f^{m_n}x_1,z_2),F^h(f^{m_n}x_1,y_2)\right)
+d\left(F^h(f^{m_n}x_1,y_2),F^h(f^{m_n}x_2,y_2)\right)\\
\leq & \iota^{-2}e^{2\sqrt\delta h} 
\left(d_Y(y_1,z_1)+d_Y(y_2,z_2)\right) +M^h d_Y(z_1,z_2)\\
&\quad\quad+d_X(f^{h+m_n}x_1,f^{h+m_n}x_2)
+C_\rho\|S_h\tau(f^{m_n}x_1)-S_h\tau(f^{m_n}x_2)\|\\
\leq & d_X(f^kx_1,f^kx_2)+2\iota^{-2}\overline r_n e^{2\sqrt\delta h}
+\epsilon^2M^{h-n} +C_\rho C_\tau r_b\\
\leq & d_X(f^kx_1,f^kx_2) +2\iota^{-2}e^{-\sqrt\delta n} +\epsilon^2 
+C_\rho C_\tau r_b\\
< & d_X(f^kx_1,f^kx_2)+\epsilon,
\end{align*}
provided that $n$ is sufficiently large.
This completely establishes property (B2), and thus
completes the proof of the lemma.
\end{proof}

Based on Lemma \ref{Mal25-Lemma 5.2}, 
we can establish the matching between leaves on $\widehat{W}^u$, 
as stated in the following lemma. 
This statement departs from \cite[Lemma 6.4]{Mal25} 
due to the fact that our fiberwise action varies along the orbits of $f$.

\begin{lemma}\label{Mal25-Lemma 6.4}
Suppose that the hypotheses of Lemma \ref{Mal25-Lemma 5.2} hold,
and let the parameters $\epsilon,\epsilon_0$ be fixed as therein.
There exists $n_0\in\mathbb N$ such that for any $n>n_0$ and
any $x,x'\in G_n$ satisfying \eqref{eq:B2-base-condition}, 
and for any $y\in T_{n,x}$ and $z\in T_{n,x'}$,
there exists a measure preserving map
$$
\theta_{(x,y),(x',z)}:
\left(W_n((x,y)),m_{W_n((x,y))}\right)
\longrightarrow
\left(W_n((x',z)),m_{W_n((x',z))}\right)
$$
and a subset $\widetilde W_n((x,y))\subset W_n((x,y))$ with
$m_{W_n((x,y))}(\widetilde W_n((x,y)))>1-3\epsilon$, 
such that for all $m_n\leq k\leq n$ and all $w\in\widetilde W_n((x,y))$,
$$
d\left(F^k((x,w)),F^k((x',\theta_{(x,y),(x',z)}w))\right) 
<d(f^kx,f^kx')+\epsilon.
$$
\end{lemma}

\begin{proof}
Fix $x,x'\in G_n$ satisfying \eqref{eq:B2-base-condition},
$y\in T_{n,x}$, and $z\in T_{n,x'}$. 
Fix the family $\{B_i^x\}_{i\in J(n,x)}$ constructed with reference base point $x$.
Apply Lemma \ref{Mal25-Lemma 5.2} (B1) first 
with $(x_1,x_2)=(x,x)$ for the plaque $W_n((x,y))$, 
and then with $(x_1,x_2)=(x,x')$ for the plaque $W_n((x',z))$. 
In both applications the blocks are $B_i^x$,
whereas the plaques are iterated by $\rho_{S_{m_n}\tau(x)}$ and
$\rho_{S_{m_n}\tau(x')}$, respectively.
Thus, we obtain two index sets 
$J'_{n;x,x}(y)$ and $J'_{n;x',x}(z)\subset J(n,x)$ such that
$$
m\left(\bigcup_{i\in J'_{n;x,x}(y)}B_i^x\right)>1-\epsilon, \quad
m\left(\bigcup_{i\in J'_{n;x',x}(z)}B_i^x\right)>1-\epsilon.
$$
We set $I:=J'_{n;x,x}(y)\cap J'_{n;x',x}(z)$. 
Since both unions use the same pairwise disjoint family
of blocks, their intersection is precisely
$\displaystyle \bigcup_{i\in I}B_i^x$.
It follows that
\begin{align}\label{main theorem-Bernoulli-(4)}
m\left(\bigcup_{i\in I}B_i^x\right)
\geq
m\left(\bigcup_{i\in J'_{n;x,x}(y)}B_i^x\right)
+m\left(\bigcup_{i\in J'_{n;x',x}(z)}B_i^x\right)-1>1-2\epsilon.
\end{align}
Furthermore, for each $i\in I$, by Lemma \ref{Mal25-Lemma 5.2} (B1), we derive
\begin{align*}
m_{W_n((x,y))}\left(\rho_{S_{m_n}\tau(x)}^{-1}B_i^x\right)
\geq\frac{m(B_i^x)}{1+\epsilon},\quad 
m_{W_n((x',z))}\left(\rho_{S_{m_n}\tau(x')}^{-1}B_i^x\right)
\geq\frac{m(B_i^x)}{1+\epsilon}.
\end{align*}
Since the conditional probabilities on these unstable plaques are nonatomic,  
for each $i\in I$, there exists a subset
$$
\widetilde B_i^x((x,y)) \subset W_n((x,y))\cap \rho_{S_{m_n}\tau(x)}^{-1}B_i^x
$$
such that
$$
m_{W_n((x,y))}\left(\widetilde B_i^x((x,y))\right)
=\frac{m(B_i^x)}{1+\epsilon}.
$$
Correspondingly, there exists a subset
$$
\widetilde B_i^{x'}((x',z))\subset W_n((x',z))\cap\rho_{S_{m_n}\tau(x')}^{-1}B_i^x
$$
such that
$$
m_{W_n((x',z))}\left(\widetilde B_i^{x'}((x',z))\right)
=m_{W_n((x,y))}\left(\widetilde B_i^x((x,y))\right)
=\frac{m(B_i^x)}{1+\epsilon}.
$$
According to Lemma \ref{DKRH24-Lemma 2.15}, 
there exists a measure preserving bijection 
$$
\theta^i_{(x,y),(x',z)}:
\left(\widetilde B_i^x((x,y)),m_{W_n((x,y))}\right)
\longrightarrow
\left(\widetilde B_i^{x'}((x',z)),m_{W_n((x',z))}\right).
$$
Clearly, for any $w\in\widetilde B_i^x((x,y))$, we have
$$
u:=\rho_{S_{m_n}\tau(x)}w\in B_i^x,\quad
v:=\rho_{S_{m_n}\tau(x')}\theta^i_{(x,y),(x',z)}w\in B_i^x.
$$
The cocycle identity yields 
$$
F^{k-m_n}(f^{m_n}x,u)=F^k(x,w),\quad
F^{k-m_n}(f^{m_n}x',v)
=F^k(x',\theta^i_{(x,y),(x',z)}w).
$$
Hence, Lemma \ref{Mal25-Lemma 5.2} (B2) implies that
for all $m_n\leq k\leq n$,
\begin{align*}
d\left(F^k((x,w)),F^k((x',\theta^i_{(x,y),(x',z)}w))\right)
<d(f^kx,f^kx')+\epsilon.
\end{align*}

Now, we define the disjoint unions
$$
\widetilde W_n((x,y))=\bigcup_{i\in I}\widetilde B_i^x((x,y)), \quad
\widetilde W_n((x',z))=\bigcup_{i\in I}\widetilde B_i^{x'}((x',z)).
$$
By taking the disjoint union of these maps
$\theta^i_{(x,y),(x',z)}$ over all $i\in I$,
we obtain a measure preserving map $\theta_{(x,y),(x',z)}$
from $\widetilde W_n((x,y))$ to $\widetilde W_n((x',z))$.
Combining \eqref{main theorem-Bernoulli-(4)}
with Lemma \ref{Mal25-Lemma 5.2} (B1), we obtain 
\begin{align*}
m_{W_n((x,y))}\left(\widetilde W_n((x,y))\right)
=\frac1{1+\epsilon}m\left(\bigcup_{i\in I}B_i^x\right)
>\frac{1-2\epsilon}{1+\epsilon}>1-3\epsilon.
\end{align*}
Finally, since $m_{W_n((x,y))}$ and $m_{W_n((x',z))}$
are both probability measures, $\theta_{(x,y),(x',z)}$
can be extended to a measure preserving map from
$W_n((x,y))\setminus\widetilde W_n((x,y))$ to
$W_n((x',z))\setminus\widetilde W_n((x',z))$.
The conditional measures of these two complements are equal and nonatomic. 
This completes the proof.
\end{proof}

\subsection{Proof of Theorem \ref{main theorem-Bernoulli}}\label{subsect:Proof of Theorem main theorem-Bernoulli}

We now apply the preceding matching construction to atoms of the
partition of the skew product and complete the proof of
Theorem \ref{main theorem-Bernoulli}. 

Suppose that $\zeta = \mu \times m$ is an ergodic $F$-invariant measure.
By \cite[Theorem 3.12]{Bow75}, we choose a Markov partition $\mathcal P$ of $X$ 
of sufficiently small diameter, 
and we let $\mathcal Q$ be a regular $D$-nice partition of $Y$. 
These partitions remain fixed throughout the argument. 
For integers $N_2\geq N_1\geq1$, we set
$$
\mathcal C_{N_1,N_2}
=\bigvee_{i=N_1}^{N_2}F^i(\mathcal P\times\mathcal Q).
$$
Let $\mathcal C_{N_1,N_2}(x,y)$ denote the atom containing $(x,y)\in X\times Y$. 
If an atom $C\in \mathcal C_{N_1,N_2}$ of positive measure is associated with 
the sequence of partition elements $(P_i, Q_i)_{i=N_1}^{N_2}$, then 
\begin{align*}
C=\{(x,y):x\in\mathcal P_0,\ y\in C_x\}, 
\end{align*}
where
\begin{align*}
\mathcal P_0=\bigcap_{i=N_1}^{N_2}f^iP_i,\quad 
C_x=\bigcap_{i=N_1}^{N_2}\rho_{S_i\tau(f^{-i}x)}Q_i=\bigcap_{i=N_1}^{N_2} \rho_{-S_{-i}\tau(x)} Q_i 
\quad \text{ for }x\in\mathcal P_0.
\end{align*}
We set $C_x=\varnothing$ outside $\mathcal P_0$. 
Indeed, $(x,y)\in F^i(P_i\times Q_i)$ if and only if
$f^{-i}x\in P_i$ and $\rho_{S_{-i}\tau(x)}y\in Q_i$. 
Consequently, the projection of $C$ onto the base is determined by the backward iterates of $P_i$, 
while its fiber sections are formed by the corresponding intersections of $Q_i$ within each fiber. 

Write $\zeta|_{C}$ for the normalized restriction of $\zeta$ to $C$, and
let $\lambda_C$ be its projection onto $X$. 
By Fubini's theorem, we deduce 
\begin{align*}
\rd\lambda_C(x) = \frac{m(C_x)}{\zeta(C)}\,\rd\mu(x),\quad
\zeta|_{C}(B) = \int_X m|_{C_x}(B_x)\,\rd\lambda_C(x),
\end{align*}
where $B_x = \{y \mid (x,y)\in B\}$ for $x\in X$.
Points $x\in X$ with $m(C_x) = 0$ have zero $\lambda_C$-measure 
and thus do not affect the integral. 
In particular, the conditional measure is $\lambda_C$, 
which is distinct from $\mu|_{\mathcal P_0}$ in general.

Next, we introduce the unstable saturation needed for the partition 
in our skew product setting.
Recall that $\widehat W^u$ denotes the unstable foliation on the fiber.
We use $\widehat W^u((x,y),r)$ to denote the ball of radius $r$ centered at
$y$ in the fiber over $x$, with respect to the induced metric on $\widehat W^u$.
Here, a vertical unstable plaque is regarded as a subset of $Y$, which is 
identified with a subset of $X \times Y$ via $\{x\} \times Y$.
We can adapt the arguments from \cite[Lemma 4.2]{DKRH24} and
\cite[Lemma 6.2]{Mal25} to our setting, yielding the following lemma:

\begin{lemma}\label{Mal25-Lemma 6.2}
Let $\mathcal Q$ be a D-Nice partition.
For any $\eta>0$, there exist $\xi>0$, $N_0\in\mathbb N$, 
and a measurable subset $E\subset X\times Y$ with $\zeta(E)>1-\eta$ 
such that for every $N_2\geq N_1\geq N_0$,
$$
\{x\}\times\widehat W^u((x,y),\xi)\subset \mathcal C_{N_1,N_2}(x,y),
\quad (x,y)\in E.
$$
\end{lemma}

\begin{proof}
By Lemma \ref{all LY} and Lemma \ref{Mal25-Lemma 3.1}, 
for the given $\eta$, there are a measurable set $E_0$ with $\zeta(E_0)>1-\eta/2$ and constants
$\xi>0$, $A<\infty$, and $\chi>0$ such that 
$$
d_Y\left(\pi_YF^{-i}(x,y),\pi_YF^{-i}(x,z)\right) \leq Ae^{-\chi i}
$$
for every $(x,y)\in E_0$, $z\in\widehat W^u((x,y),\xi)$, and $i\geq0$.

Since the partition $\mathcal Q$ is $D$-nice and $\zeta$ is $F$-invariant, we have 
$$
\zeta\{(x,y):
d_Y(\pi_YF^{-i}(x,y),\partial\mathcal Q)
\leq2Ae^{-\chi i}\}
\leq2DAe^{-\chi i}
$$
for sufficiently large $i$. Choose $N_0$ so that the sum of
these bounds over $i\geq N_0$ is less than $\eta/2$,
 and remove the corresponding sets from $E_0$. 
The remaining set $E$ satisfies $\zeta(E)>1-\eta$. 

If $(x,y)\in E$ and $z\in\widehat W^u((x,y),\xi)$, 
then the fiber coordinates of $F^{-i}(x,y)$ and $F^{-i}(x,z)$ 
belong to the same $\mathcal Q$-atom for every $i\geq N_0$. 
Their base coordinates are both $f^{-i}x$, 
so they also have the same $\mathcal P$-name. 
Hence $(x,z)$ and $(x,y)$ belong to the same $\mathcal C_{N_1,N_2}$-atom 
whenever $N_2\geq N_1\geq N_0$. 
\end{proof}

For a finite measure $\nu$ and a measurable set $A$ of positive measure, 
we write $\nu_{|A}=\nu(A)^{-1}\nu|_A$.
For probability measures $\nu_1,\nu_2$ on the same space, we denote 
$$
\|\nu_1-\nu_2\|_{\mathrm{TV}}=\sup_B|\nu_1(B)-\nu_2(B)|,
$$
where the supremum is taken over all measurable sets $B$.

\begin{corollary}\label{Mal25-Corollary 6.3}
Under the conclusions of Lemma \ref{Mal25-Lemma 5.2}, 
we take the set $\mathcal K_n^x$ with
$m(\mathcal K_n^x)>1-\epsilon^2$ for $x\in G_n$ and
$\epsilon_0=\epsilon^2/100$. We define 
$$
\mathcal K_n=\{(x,y):x\in G_n,\ y\in\mathcal K_n^x\}.
$$
For $x\in G_n$, we write the plaque disintegration as
\begin{align}\label{eq:plaque-disintegration}
m(B\cap\mathcal K_n^x)=\int_{T_{n,x}}m_{W_n((x,v))}(B)\,\rd\nu_{n,x}(v),
\end{align}
where $\nu_{n,x}$ is a finite nonatomic factor measure.

For all sufficiently large $n$ and all
$N_2\geq N_1\geq N_0$, every atom $C \in \mathcal C_{N_1,N_2}$ 
has a measurable subset $K_C\subset C\cap\mathcal K_n$ 
that consists of plaques, satisfying 
$$
\sum_{C\in \mathcal C_{N_1,N_2}}\zeta(C\setminus K_C)<3\epsilon^2.
$$
Consequently, except for a collection of atoms of total measure less
than $3\epsilon$, every positive measure atom satisfies 
\begin{align}\label{Mal25-Corollary 6.3-(2)}
k_C:=\frac{\zeta(K_C)}{\zeta(C)}>1-\epsilon.
\end{align}
For such an atom, let $\widehat\lambda_C$ be the base projection of $\zeta_{|K_C}$. Then
\begin{align}\label{eq:core-base-TV}
\|\widehat\lambda_C-\lambda_C\|_{\mathrm{TV}}\leq 1-k_C<\epsilon.
\end{align}
Moreover, setting 
$$
T_{C,x}=\{v\in T_{n,x}:m_{W_n((x,v))}(C_x)=1\},
$$
the restrictions to $K_C$ and to its fiber sections inherit the
disintegrations by restricting \eqref{eq:plaque-disintegration} to $T_{C,x}$. 
In particular, any normalized restriction of these factor measures 
to a set of positive measure is nonatomic. 
\end{corollary}

\begin{proof}
The disintegration \eqref{eq:plaque-disintegration} 
follows from Lemma \ref{Mal25-Lemma 5.2} and the Rokhlin disintegration theorem. 
We define $K_C$ as the intersection of $C$ with 
the union of the plaques indexed by $T_{C,x}$. 

Apply Lemma \ref{Mal25-Lemma 6.2} with $\eta=\epsilon^2$.
By Lemma \ref{Mal25-Lemma 5.2} (U2), for all sufficiently large $n$, 
every plaque meeting the set $E$ of Lemma \ref{Mal25-Lemma 6.2} 
is contained in an atom of $\mathcal C_{N_1,N_2}$. 
Hence almost every point of $\mathcal K_n\cap E$ belongs to the corresponding $K_C$. 
Since $\mu(G_n^c)<2\epsilon_0$, it follows that
$$
\sum_C\zeta(C\setminus K_C)
\leq\zeta(\mathcal K_n^c)+\zeta(E^c)
<2\epsilon_0+2\epsilon^2<3\epsilon^2.
$$
Combining this with Markov's inequality yields the desired estimate \eqref{Mal25-Corollary 6.3-(2)}. 
For $k_C < 1$ (with the case $k_C = 1$ being immediate), the identity 
$$
\zeta_{|C}=k_C\zeta_{|K_C}+(1-k_C)\zeta_{|C\setminus K_C}
$$
yields \eqref{eq:core-base-TV} after taking base projections.

It remains to prove that $\nu_{n,x}$ is nonatomic. 
Setting $\mathbf a=\mu(\tau)\neq\mathbf 0$, we define
$$
V_{\mathbf a}(y)=\left.\frac{d}{ds}\right|_{s=0}\rho_{s\mathbf a}(y).
$$
Local freeness implies that $V_{\mathbf a}$ is nowhere zero. 
According to Lemma \ref{all LY}, the Lyapunov exponent of this vector is zero. 
By Lemma \ref{Mal25-Lemma 5.2} (U2), each plaque $W=W_n((x,v))$ is contained in a fiber unstable manifold $\widehat W^u$. Its tangent vectors contract under backward iteration, whereas the preceding bound prevents $V_{\mathbf a}$ from contracting. Hence
$$
V_{\mathbf a}(y)\notin T_yW
\quad\text{for }y\in W.
$$
Consequently,the differential of the map $(s,z)\mapsto\rho_{s\mathbf a}(z)$ 
is injective at $(0,y)$ for $y\in W$. 
By the constant rank theorem, 
$W$ can be covered by countably many small measurable pieces $P$ 
such that for each $P$, the map 
$$
(-\delta,\delta)\times P\longrightarrow Y,
\quad
(s,z)\mapsto\rho_{s\mathbf a}(z),
$$
is one-to-one for some $\delta>0$ depending on $P$. 
For any distinct $s_1,\ldots,s_N\in(-\delta,\delta)$, 
the sets $\rho_{s_i\mathbf a}(P)$ are pairwise disjoint. 
Then, we have 
$$
N\,m(P)=m\left(\bigsqcup_{i=1}^{N}\rho_{s_i\mathbf a}(P)\right)\leq 1.
$$
Since $N$ is arbitrary, $m(P)=0$. Countable additivity then yields $m(W)=0$. 

We proceed by contradiction and assume that $\nu_{n,x}$ has an atom at some 
$v_0\in T_{n,x}$, so that $\nu_{n,x}(\{v_0\})>0$. 
Let $W_0=W_n((x,v_0))$ be the corresponding plaque. 
Since $W_0\subset\mathcal K_n^x$, 
by applying \eqref{eq:plaque-disintegration} with $B=W_0$, we obtain 
$$
m(W_0)=\int_{T_{n,x}}m_{W_n((x,v))}(W_0)\,\rd\nu_{n,x}(v).
$$
Since the integrand is nonnegative and $m_{W_0}(W_0)=1$, we yields 
$$
\begin{aligned}
m(W_0)\geq \int_{\{v_0\}} m_{W_n((x,v))}(W_0)\,\rd\nu_{n,x}(v)
=m_{W_0}(W_0)\,\nu_{n,x}(\{v_0\})= \nu_{n,x}(\{v_0\})>0,
\end{aligned}
$$
which contradicts the fact that $m(W_0)=0$. 
Thus, every plaque has zero $m$-measure. 
This implies that $\nu_{n,x}$ is nonatomic.

Finally, for every measurable $A\subset T_{n,x}$ with $\nu_{n,x}(A)>0$ and every $v \in T_{n,x}$, 
$$
(\nu_{n,x})_{|A}(\{v\})=\frac{\nu_{n,x}(A\cap\{v\})}{\nu_{n,x}(A)}=0.
$$
Hence all normalized restrictions of $\nu_{n,x}$, including those to $T_{C,x}$, are nonatomic. 
\end{proof}

The next lemma establishes the base matching for these base conditional measures 
without modifying the underlying Markov partition.

\begin{lemma}\label{lem:past-atom-base-matching}
Fix $q\geq1$, and we define 
$\displaystyle \mathcal R_{n,q}=\bigvee_{j=-q}^{n+q}f^{-j}\mathcal P$. 
For every $\eta>0$ there exists $N_0$ such that, for all
$n\geq1$ and $N_2\geq N_1\geq N_0$,
\begin{align}\label{eq:past-base-average}
\sum_{\substack{C\in\mathcal C_{N_1,N_2}\\\zeta(C)>0}}
\zeta(C)\,\frac{1}{2} \sum_{R\in\mathcal R_{n,q}}|\lambda_C(R)-\mu(R)|<\eta.
\end{align}
The choice of $N_0$ is independent of $n$, $N_1$, and $N_2$.
\end{lemma}

\begin{proof}
Recall that the Markov coding
$$
\pi:(\Sigma_A,\sigma,\nu)\longrightarrow(X,f,\mu),
\quad
\pi_*\nu=\mu,
$$
is a H\"older continuous map, 
where $\nu$ is also the Gibbs measure of the lifted H\"older potential.
Outside a $\nu$-null set, $\pi$ is one-to-one,
and the symbolic itinerary agrees with the $\mathcal P$-itinerary;
see \cite[proof of Theorem 4.1]{Bow75}.
We define $\widehat\tau:=\tau\circ\pi$, which is also H\"older continuous. 

Fix $q<M<N_1$. Let $\mathcal A$ be the alphabet of $\Sigma_A$.
For each $a\in\mathcal A$, fix a one-sided admissible sequence
$\omega^+(a)=\bigl(\omega_\ell^+(a)\bigr)_{\ell\geq0}$ 
such that $\omega_0^+(a)=a$. For $\xi\in\Sigma_A$, define
$\bar\xi\in\Sigma_A$ by
$$
\bar\xi_j=
\begin{cases}
\xi_j, & j\leq-M,\\
\omega_{j+M}^+(\xi_{-M}), & j>-M.
\end{cases}
$$
Thus, $\bar\xi$ agrees with $\xi$ on all coordinates $j\leq-M$,
while its coordinates after $-M$ are given by the fixed admissible
future associated with the symbol $\xi_{-M}$. 
By the H\"older continuity of $\widehat\tau$, there exist
$H<\infty$ and $0<\kappa<1$ such that, for every $j>M$,
$$
\left\|\widehat\tau(\sigma^{-j}\xi)-\widehat\tau(\sigma^{-j}\bar\xi)\right\|
\leq H\kappa^{j-M}.
$$
Consequently, the following series converges absolutely:
$$
b_M(\xi)=\sum_{j=1}^{\infty}
\left(\widehat\tau(\sigma^{-j}\xi)-\widehat\tau(\sigma^{-j}\bar\xi)\right).
$$
We set $\bar y=\rho_{-b_M(\xi)}y$. 
Since the action $\rho_{\bf t}$ preserves $m$, the transformation
$(\xi,y)\mapsto(\xi,\bar y)$ preserves $\nu\times m$.
Hence, in the new coordinates $(\xi,\bar y)$, the joint
has distribution $\nu\times m$. 

For $i\geq N_1$, we define 
$$
A_i(\xi)=\sum_{j=1}^{i}\widehat\tau(\sigma^{-j}\xi), \quad
r_i(\xi)=b_M(\xi)-A_i(\xi)+A_i(\bar\xi).
$$
Then, we obtain $\rho_{-A_i(\xi)}y=\rho_{r_i(\xi)}\left(\rho_{-A_i(\bar\xi)}\bar y\right) $ 
and 
$$
\|r_i(\xi)\|=\left\|\sum_{j=i+1}^{\infty}
\left(\widehat\tau(\sigma^{-j}\xi)
-\widehat\tau(\sigma^{-j}\bar\xi)\right)\right\|
\leq
\frac{H}{1-\kappa}\kappa^{i-M}.
$$
Therefore, by \eqref{gruop action bound}, there exists
$C_0<\infty$ such that
$$
d_Y\left(
\rho_{-A_i(\xi)}y,
\rho_{-A_i(\bar\xi)}\bar y
\right)
\leq C_0\kappa^{i-M}.
$$

Let $\mathsf A$ denote the $(\mathcal P\times\mathcal Q)$-name of $(\xi,y)$ 
at times $-N_2,\ldots,-N_1$, 
and let $\bar{\mathsf A}$ denote the corresponding name obtained from $(\bar\xi,\bar y)$.
Since $N_1>M$, the two names have the same base symbols at all these times.
As the partition $\mathcal Q$ is $D$-nice,
if $N_1-M$ is sufficiently large, we have 
\begin{align}\label{eq:past-name-error}
(\nu\times m)\{\mathsf A\neq\bar{\mathsf A}\}
\leq DC_0\sum_{i=N_1}^{N_2}\kappa^{i-M}
\leq \frac{DC_0}{1-\kappa}\kappa^{N_1-M},
\end{align}

Let $\mathsf B$ be the $\mathcal P$-name on $[-q,n+q]$.
For the fixed $\bar y$, $\bar{\mathsf A}$ depends only on coordinates at times at most $-M$,
whereas $\mathsf B$ depends only on coordinates at times at least $-q$.
The estimate in \cite[proof of Theorem 1.25]{Bow75} gives constants
$C_1<\infty$ and $0<\beta<1$ such that 
$$
\frac{1}{2}\sum_{a,v} 
\left| (\nu\times m)(\bar{\mathsf A}=a,\mathsf B=v)
-(\nu\times m)(\bar{\mathsf A}=a)(\nu\times m)(\mathsf B=v) \right|
\leq C_1\beta^{M-q}.
$$
This estimate is uniform in $n$ and $N_2$.
Since this estimate is uniform in $\bar y$, integrating with
respect to $m$ gives the same bound for $\nu\times m$. 

By \eqref{eq:past-name-error}, the joint distributions of
$(\mathsf A,\mathsf B)$ and $(\bar{\mathsf A},\mathsf B)$ differ
by at most $C_2\kappa^{N_1-M}$ in total variation, 
and the same bound holds for the marginal distributions of $\mathsf A$ and $\bar{\mathsf A}$. 
Hence, the triangle inequality gives 
\begin{align*}
\frac{1}{2}\sum_{a,v}
\left|(\nu\times m)(\mathsf A=a,\mathsf B=v)
-(\nu\times m)(\mathsf A=a)(\nu\times m)(\mathsf B=v)\right|
\leq \frac{2DC_0}{1-\kappa} \kappa^{N_1-M} + C_1 \beta^{M-q}.
\end{align*}
All constants are independent of $n,N_1,N_2$, and $M$.
Under the coding, the left-hand side is precisely the left-hand
side of \eqref{eq:past-base-average}.
Taking $M=\lfloor N_1/2\rfloor$ and then choosing $N_0$
sufficiently large proves the lemma.
\end{proof}

We are now ready to present the proof of Theorem \ref{main theorem-Bernoulli}.

\begin{proof}[Proof of Theorem \ref{main theorem-Bernoulli}]
Fix $0<\epsilon<1/100$, $\delta=\epsilon^{100^{100}}$,
and $\epsilon_0=\epsilon^2/100$.

Retain the fixed partitions $\mathcal P,\mathcal Q$ above.
Choose $r_b$ as in
\eqref{eq:B2-base-radius}. There is an integer $q\geq1$ such that
every atom of $\bigvee_{j=-q}^{q}f^{-j}\mathcal P$ has diameter
less than $r_b$. Consequently, points in the same atom of
$\mathcal R_{n,q}$ have distance less than $r_b$ at every time
$0,\ldots,n$.

Choose $\xi_n\to0$ satisfying the two rate conditions in Lemma \ref{Mal25-Lemma 5.2}. 
For every sufficiently large $n$ and every $x\in G_n$, 
let $\mathcal K_n^x$ be the set obtained from the construction 
in the proof of that lemma such that $m(\mathcal K_n^x)>1-\epsilon^2$. 
Increase $N_0$ and $n_0$ so that Corollary
\ref{Mal25-Corollary 6.3} and Lemma
\ref{lem:past-atom-base-matching} apply, with the right-hand side
of \eqref{eq:past-base-average} less than $\epsilon^2$ whenever
$N_2\geq N_1\geq N_0$ and $n\geq n_0$.

Fix $N_2\geq N_1\geq N_0$ and $n\geq n_0$. 
Remove the atoms for which $k_C\leq1-\epsilon$ or 
$$
\frac{1}{2}\sum_{R\in\mathcal R_{n,q}}|\lambda_C(R)-\mu(R)|\geq\epsilon.
$$
Corollary \ref{Mal25-Corollary 6.3}, Lemma
\ref{lem:past-atom-base-matching}, and Markov's inequality show
that the removed atoms have total $\zeta$-measure less than
$4\epsilon$. Fix a remaining atom $C$. By \eqref{eq:core-base-TV}, we have 
$$
\frac{1}{2}\sum_{R\in\mathcal R_{n,q}}|\widehat\lambda_C(R)-\mu(R)|<2\epsilon.
$$

The probabilities $\widehat\lambda_C$ and $\mu$ are nonatomic.
The matching argument used in the proof of Lemma \ref{Mal25-Lemma 6.4}, 
together with Lemma \ref{DKRH24-Lemma 2.15}, 
therefore gives a measure preserving isomorphism 
$$
\theta_C:(X,\widehat\lambda_C)\longrightarrow(X,\mu), \quad 
(\theta_C)_*\widehat\lambda_C=\mu,
$$
and a set $E_C^0$ with $\widehat\lambda_C(E_C^0)>1-2\epsilon$ 
such that $x$ and $\theta_C(x)$ belong to the same atom 
of $\mathcal R_{n,q}$ for $x\in E_C^0$. 
Indeed, the construction matches a mass of $\min\{\widehat\lambda_C(R),\mu(R)\}$ 
within each atom $R\in\mathcal R_{n,q}$, yielding 
\[
\widehat\lambda_C(E_C^0)
=\sum_{R\in\mathcal R_{n,q}} \min\{\widehat\lambda_C(R),\mu(R)\}
=1-\frac{1}{2}\sum_{R\in\mathcal R_{n,q}}|\widehat\lambda_C(R)-\mu(R)|
>1-2\epsilon.
\]
Since $K_C\subset\mathcal K_n$, its base
projection is supported on $G_n$. We set 
$$
E_C=E_C^0\cap G_n\cap\theta_C^{-1}(G_n).
$$
Then, we have 
$$
\widehat\lambda_C(E_C)>1-2\epsilon-2\epsilon_0,
$$
and every pair $x$ and $x'=\theta_C(x)$ with $x\in E_C$ lies in
$G_n\times G_n$ and satisfies
\begin{align*}
d_X(f^jx,f^jx')<r_b,\quad 0\leq j\leq n.
\end{align*}

Recall that $(K_C)_x:=\{y\in Y:(x,y)\in K_C\}$. 
Since $K_C$ is the union of the plaques indexed by $T_{C,x}$, the disintegration
\eqref{eq:plaque-disintegration} gives
$$
m((K_C)_x)=\int_{T_{n,x}}\mathbf 1_{T_{C,x}}(v)\,\rd\nu_{n,x}(v)
=\nu_{n,x}(T_{C,x}).
$$
Moreover, since $\widehat\lambda_C$ is the projection of
$\zeta_{|K_C}$ onto $X$, we have
$$
\rd\widehat\lambda_C(x)
=
\frac{m((K_C)_x)}{\zeta(K_C)}\,\rd\mu(x).
$$
Consequently, $m((K_C)_x)>0$ for
$\widehat\lambda_C$-almost every $x$. For such $x$, we set
$$
a_C(x):=m((K_C)_x)=\nu_{n,x}(T_{C,x})>0.
$$
For $x\in E_C$, put $x'=\theta_C(x)$ and
$b_n(x')=m(\mathcal K_n^{x'})>1-\epsilon^2$. 
Combining Corollary \ref{Mal25-Corollary 6.3} with \eqref{eq:plaque-disintegration}, we obtain 
\begin{align*}
m_{|(K_C)_x}=\int_{T_{C,x}}m_{W_n((x,v))}\,\frac{\rd\nu_{n,x}(v)}{a_C(x)}, \quad
m_{|\mathcal K_n^{x'}}=\int_{T_{n,x'}}m_{W_n((x',v'))}\,\frac{\rd\nu_{n,x'}(v')}{b_n(x')},
\end{align*} 
where by Corollary \ref{Mal25-Corollary 6.3}, the two probability measures
$\displaystyle \frac{\nu_{n,x}|_{T_{C,x}}}{a_C(x)}$ and 
$\displaystyle \frac{\nu_{n,x'}}{b_n(x')}$ are nonatomic.

By Lemma \ref{DKRH24-Lemma 2.15}, there exists a measure preserving isomorphism 
$$
\Theta_C^x: \left(T_{C,x},\frac{\nu_{n,x}|_{T_{C,x}}}{a_C(x)}\right)
\longrightarrow \left(T_{n,x'},\frac{\nu_{n,x'}}{b_n(x')}\right).
$$
By Lemma \ref{Mal25-Lemma 6.4}, for each $v\in T_{C,x}$, 
there exists a measure preserving map 
$$
\theta_{(x,v),(x',\Theta_C^x(v))}: 
W_n((x,v))\longrightarrow W_n((x',\Theta_C^x(v))).
$$
Joining these maps defines a bijection (modulo null sets) 
$$
\psi_{C,x}: ((K_C)_x,m_{|(K_C)_x})
\longrightarrow (\mathcal K_n^{x'},m_{|\mathcal K_n^{x'}})
$$
which is exactly measure preserving for these normalized fiber probabilities. 
There exists a subset $D_{C,x}\subset(K_C)_x$ 
with $m_{|(K_C)_x}(D_{C,x})>1-3\epsilon$ such that
$$
d\left(F^j(x,y),F^j(x',\psi_{C,x}(y))\right)
<d_X(f^jx,f^jx')+\epsilon<2\epsilon
$$
for every $y\in D_{C,x}$ and $m_n\leq j\leq n$. 

For every Borel set $B\subset Y$, 
the map $(x,v)\mapsto m_{W_n((x,v))}(B)$ is measurable.
Using the nonatomicity established above and
\cite[Theorem 2.3]{MPv}, we choose the maps $\Theta_C^x$
and the maps supplied by Lemma \ref{Mal25-Lemma 6.4} such that
$$
(x,v)\mapsto\Theta_C^x(v)
\quad \text{and}\quad 
(x,v,w)\mapsto \theta_{(x,v),(\theta_C(x),\Theta_C^x(v))}(w)
$$
are measurable on their domains, respectively. 
Since the plaque partition is measurable, 
combining these maps gives a measurable map
$$
(x,y)\mapsto\psi_{C,x}(y)\quad \text{on}\quad 
\{(x,y):x\in E_C,\ y\in(K_C)_x\}.
$$
On this domain, we define
$$
\Theta_C(x,y):=\left(\theta_C(x),\psi_{C,x}(y)\right),
$$
and extend it measurably to the rest of $C$.
Since $\theta_C$ and the fiber maps $\psi_{C,x}$ are
isomorphisms modulo null sets, $\Theta_C$ is injective
modulo null sets on its domain.

The preceding measurable choices also ensure that the set 
$$
D_C:=\{(x,y):x\in E_C,\ y\in D_{C,x}\}
$$
is measurable. 
Since $D_C\subset K_C\subset C$, by \eqref{Mal25-Corollary 6.3-(2)}, we have
\begin{align*}
\zeta_{|C}(D_C)
=\frac{\zeta(D_C)}{\zeta(C)}
=\frac{\zeta(K_C)}{\zeta(C)}\frac{\zeta(D_C)}{\zeta(K_C)}
=k_C\,\zeta_{|K_C}(D_C).
\end{align*} 
By the disintegration of $\zeta_{|K_C}$ over
$\widehat\lambda_C$ and the definition of $D_C$, we have 
\begin{align*}
\zeta_{|K_C}(D_C)
=\int_{E_C} m_{|(K_C)_x}(D_{C,x})\,\rd\widehat\lambda_C(x)
>(1-3\epsilon)\widehat\lambda_C(E_C)
>(1-3\epsilon)(1-2\epsilon-2\epsilon_0).
\end{align*}
Since $k_C>1-\epsilon$ and
$\epsilon_0=\epsilon^2/100$, we obtain
\begin{align*}
\zeta_{|C}(D_C)>k_C (1-3\epsilon)(1-2\epsilon-2\epsilon_0)
>(1-\epsilon)(1-2\epsilon-2\epsilon_0)(1-3\epsilon)>1-7\epsilon.
\end{align*}

For every measurable subset $B\subset D_C$, the disintegrations yield 
\begin{align*}
\zeta(\Theta_C(B))&=\int_{E_C}b_n(\theta_C(x)) m_{|(K_C)_x}(B_x)\,\rd\widehat\lambda_C(x),\\
\zeta_{|C}(B)&=k_C\int_{E_C}m_{|(K_C)_x}(B_x)\,\rd\widehat\lambda_C(x).
\end{align*}
Since $1-\epsilon^2<b_n(\theta_C(x))\leq 1$ and $1-\epsilon<k_C\leq1$, the relative error is less than
$2\epsilon$. Together with the estimate on $C\setminus D_C$,
this makes $\Theta_C$ an $8\epsilon$-measure preserving map from
$(C,\zeta_{|C})$ to $(X\times Y,\zeta)$.

For every $(x,y)\in D_C$, the orbit estimate above gives 
$$
\frac1n\#\{0\leq j<n:
d(F^j(x,y),F^j\Theta_C(x,y))<2\epsilon\}
\geq\frac{n-m_n}{n}\geq1-\epsilon>1-8\epsilon.
$$
Thus Corollary \ref{DKRH24-Corollary 2.21}, with its parameter
replaced by $8\epsilon$, shows that
$\mathcal P\times\mathcal Q$ is very weak Bernoulli. All thresholds
are uniform in $N_2\geq N_1$ and in $n$.

Finally, applying the same argument to Markov partitions on $X$ 
and regular $D$-nice fiber partitions on $Y$ whose diameters tend to zero, 
we observe that these product partitions converge to the partition into points. 
Since $\epsilon$ can be arbitrarily small, 
combining this result with Theorem \ref{DKRH24-Theorem 2.18}, 
we conclude that the system $(X\times Y,F,\mu\times m)$ is Bernoulli.
\end{proof}

\section{Proof of Theorem \ref{main theorem-mixing}}\label{sect:main2}

In this section, we prove Theorem \ref{main theorem-mixing}.
Unlike the Bernoulli argument in the preceding section, 
the proof uses the statistical criteria for skew products established in \cite{DDKN22-1}.

\begin{proof}[Proof of Theorem \ref{main theorem-mixing}]
Since the base measure $\mu$ is the equilibrium state of some H\"older potential 
for the topologically mixing Anosov diffeomorphism $f$, 
by \cite[Proposition 3.1]{KS01}, 
 it is exponentially mixing of all orders.

 To apply \cite[Theorems 4.1(a) and 4.15]{DDKN22-1}, 
it remains to verify the exponential large deviation bounds for $\tau$ 
in the sense of \cite[Definition 3.2]{DDKN22-1}. 
Namely, for every $\epsilon>0$,
we need to show that there exist constants $C,\delta>0$ such that
$$
\mu\left(\left\{ x\in X: \left\| \frac{S_N\tau(x)}{N}-\mu(\tau) \right\|
\geq\epsilon \right\} \right)\leq Ce^{-\delta N}
$$
for every $N\in\mathbb N$.
We first establish this large deviation estimate.

If $\tau$ is cohomologous to a constant, then there exist $c\in\mathbb R^k$ 
and a continuous function $P:X\to\mathbb R^k$ such that
$$
\tau(x)=c+P(fx)-P(x).
$$
Since $\mu$ is $f$-invariant, we have $\mu(\tau)=c$. 
Summing along the orbit, 
$$
S_N\tau(x) = Nc+P(f^Nx)-P(x),
$$
and hence
$$
\left\| \frac{S_N\tau(x)}{N}-\mu(\tau) \right\|
\leq \frac{2\|P\|_\infty}{N}.
$$
Therefore, for every $\epsilon>0$,
$$
\left\{ x\in X: \left\| \frac{S_N\tau(x)}{N}-\mu(\tau) \right\|
\geq\epsilon \right\} =\varnothing
$$
for all sufficiently large $N$. 
Thus the exponential large deviation bounds hold trivially in this case.

We now assume that $\tau$ is not cohomologous to a constant. 
Since the invariant measure $\mu$ has the Gibbs property 
(see \cite[Lemma 20.3.4]{KH95} for a proof), 
it satisfies the conditions of \cite[Proposition 3.2]{Kif90}. 
By \cite[Theorem 3.4]{Kif90}, 
for any closed subset $K$ of $\mathcal{P}$,
the space of Borel probability measures on $X$, 
equipped with weak* topology, we have 
\begin{align}\label{main theorem-mixing-(1)}
\varlimsup_{n\to \infty} \frac{1}{n}\log  \mu 
\left\{x\in X: \frac{1}{n}\sum_{i=0}^{n-1}\delta_{f^ix}\in K\right\}
\leq -\inf \left\{I(\nu):\nu\in K\right\}, 
\end{align}
where 
\begin{align*}
I(\nu) = 
\begin{cases}
P_{\top}(f,\varphi)-\int \varphi \, \rd\nu-h_\nu(f)
& \text{if $\nu$ is $f$-invariant},\\
\infty
& \text{otherwise}.
\end{cases}
\end{align*} 
For each $\epsilon>0$, define the closed subset
$K_\epsilon\subset\mathcal P$ by
$$
K_\epsilon := 
\left\{ \nu\in\mathcal P: \|\nu(\tau)-\mu(\tau)\|\geq\epsilon \right\}.
$$
Since $I(\nu)=\infty$ whenever $\nu$ is not $f$-invariant, we have
$$
\inf_{\nu\in K_\epsilon}I(\nu) 
=\inf_{\nu\in K_\epsilon\cap\mathcal M(f,X)}I(\nu).
$$
Here $\mathcal M(f,X)$ denotes the set of $f$-invariant Borel probability measures, and $P_{\top}(f,\varphi)$ is the topological pressure.
Since $\mu$ is the unique equilibrium state of $\varphi$,
$I$ vanishes on $\mathcal M(f,X)$ only at $\mu$. 
Moreover, $\mu\notin K_\epsilon$. 
Since $K_\epsilon\cap\mathcal M(f,X)$ is compact and 
$I$ is lower semicontinuous,
$$
\delta_\epsilon := \inf_{\nu\in K_\epsilon}I(\nu)
=\inf_{\nu\in K_\epsilon\cap\mathcal M(f,X)}I(\nu) >0.
$$
Combining this with \eqref{main theorem-mixing-(1)} and the identity
$$
\frac{1}{N}\sum_{i=0}^{N-1}\delta_{f^ix}(\tau)
=\frac{1}{N}\sum_{i=0}^{N-1}\tau(f^ix)
=\frac{1}{N}S_N\tau(x),
$$
we conclude that there exist constants $C,\delta>0$ such that
for every $N\in\mathbb N$,
$$
\mu\left( \left\{ x\in X:
\left\| \frac{S_N\tau(x)}{N}-\mu(\tau) \right\|\geq\epsilon 
\right\}  \right) \leq Ce^{-\delta N}.
$$
Indeed, \eqref{main theorem-mixing-(1)} gives this estimate 
for all sufficiently large $N$ with any $0<\delta<\delta_\epsilon$, 
and $C$ can be enlarged to cover the finitely many remaining values of $N$.

Therefore, in both cases, $\tau$ satisfies the exponential large
deviation bounds required in \cite[Definition 3.2]{DDKN22-1}. 
 Since $\mu$ is exponentially mixing of all orders for $f$,
$m$ is exponentially mixing of all orders for
$\{\rho_{\bf t}\}$ by assumption, 
and $\tau$ satisfies the exponential large deviation bounds with nonzero drift,
\cite[Theorem 4.15]{DDKN22-1} implies that $\mu\times m$ 
is exponentially mixing of all orders for $F$.
The Central Limit Theorem is a direct consequence of
\cite[Corollary 5.2]{DDKN22-1}. 
This establishes the result for the case of exponential mixing of all orders.

Similarly, a slightly modified argument invoking
\cite[Theorem 4.1 (a)]{DDKN22-1}
yields the proof for the exponential mixing case.
\end{proof}

\text{\quad}\\

\textsc{Chern Institute of Mathematics and LPMC, Nankai University, Tianjin 300071 China}

{Email: dongchg@nankai.edu.cn}

\textsc{Chern Institute of Mathematics and LPMC, Nankai University, Tianjin 300071 China}

{Email: qiujieqiao@mail.nankai.edu.cn}


\begin{thebibliography}{99}

\bibitem{BORH23-2} S. Ben Ovadia, and F. Rodriguez Hertz, {\it Exponential volume limits}. arXiv:2308.03910v2.

\bibitem{BEG20} M. Bj\"orklund, M. Einsiedler, and A. Gorodnik, {\it Quantitative multiple mixing}, J. Eur. Math. Soc. (JEMS) {\bf 22} (2020), no. 5, 1475--1529.

\bibitem{Bow75} R. Bowen, {\it Equilibrium states and the ergodic theory of Anosov diffeomorphisms}. Lecture Notes in Math. {\bf 470}, Springer, Berlin, 1975.

\bibitem{BE17} O. Butterley, and P. Eslami, {\it Exponential mixing for skew products with discontinuities}, Trans. Amer. Math. Soc. {\bf 369} (2017), no. 2, 783--803. 


\bibitem{Dol02} D. Dolgopyat, {\it On mixing properties of compact group extensions of hyperbolic systems}, Israel J. Math. {\bf 130} (2002), 157--205.

\bibitem{DDKN22-1} D. Dolgopyat, C. Dong, A. Kanigowski, and P. N\'andori, {\it Mixing properties of generalized {$ T,T^{-1 }$} transformations}, Israel J. Math. {\bf 247} (2022), no. 1, 21--73.

\bibitem{DDKN22-2} D. Dolgopyat, C. Dong, A. Kanigowski, and P. N\'andori, {\it Flexibility of statistical properties for smooth systems satisfying the central limit theorem}, Invent. Math. {\bf 230} (2022), no. 1, 31--120.

\bibitem{DKRH24} D. Dolgopyat, A. Kanigowski, and F. Rodriguez Hertz, {\it Exponential mixing implies Bernoulli}, Ann. of Math. (2) {\bf 199} (2024), no. 3, 1225--1292.

\bibitem{Don} C. Dong, {\it On ergodic properties of time changes of partially hyperbolic homogeneous flows}, J. Mod. Dyn., (19) (2023), 541--559.

\bibitem{DK22} C. Dong, and A. Kanigowski, {\it Bernoulli property for certain skew products over hyperbolic systems}, Trans. Amer. Math. Soc. {\bf 375} (2022), no. 3, 1607--1628.

\bibitem{DQ-1} C. Dong, and Q. Qiao, {\it Finiteness of ergodic measures of maximal entropy for certain skew product systems}, preprint (2026).

\bibitem{Dou26} M. Doucette, {\it The Bernoulli property for $PSL_2\mathbb{R}$ skew products over exponentially mixing base}, arXiv:2609.08017. 

 \bibitem{Kan18} A. Kanigowski, {\it Bernoulli property for homogeneous systems}, arXiv 1812.03209. 

\bibitem{KH95} A. Katok, and B. Hasselblatt, {\it Introduction to the Modern Theory of Dynamical Systems}, Encyclopedia of Mathematics and its Applications {\bf 54}, Cambridge, 1995.

\bibitem{Kif90} Y. Kifer, {\it Large deviations in dynamical systems and stochastic processes}, Trans. Amer. Math. Soc. {\bf 321} (1990), no. 2, 505--524.

\bibitem{KS01} M. Kotani, and T. Sunada, {\it The pressure and higher correlations for an Anosov diffeomorphism}, Ergodic Theory Dynam. Systems {\bf 21} (2001), no. 3, 807--821.


\bibitem{Mal25} A. Maldonado, {\it Exponential mixing SRB measures are Bernoulli}, arXiv:2504.04044v2.


\bibitem{MPv} R. D. Mauldin, D. Preiss, and H. von Weizs\"acker, {\it Orthogonal transition kernels}, Ann. Probab. {\bf 11} (1983), no. 4, 970--988. 

\bibitem{OW73} D. S. Ornstein, and B. Weiss, {\it Geodesic flows are Bernoullian}, Israel J. Math. {\bf 14} (1973), 184--198.

\bibitem{Oseledec} V. I. Oseledec, {\it A multiplicative ergodic theorem. {C}haracteristic {L}japunov, exponents of dynamical systems}, Trudy Moskov. Mat. Ob\v s\v c. {\bf 19} (1968), 179--210.

\bibitem{PP90} W. Parry, and M. Pollicott, {\it Zeta functions and the periodic orbit structure of hyperbolic dynamics}, Astérisque No. 187-188 (1990), 268 pp.

\bibitem{Ro52} V. A. Rohlin, {\it On the fundamental ideas of measure theory}, Amer. Math. Soc. Translation {\bf 1952} (1952), no. 71, 55 pp.

\bibitem{FMT03} M. Field, I. Melbourne, and A. T\"or\"ok, {\it Decay of correlations, central limit theorems and approximation by Brownian motion for compact Lie group extensions}, Ergodic Theory Dynam. Systems {\bf 23} (2003), no. 1, 87--110. 

\end{thebibliography}
\end{document}